\documentclass[a4paper,english]{jnsao}
\usepackage[utf8]{inputenc}
\usepackage[english]{babel}
\usepackage{graphicx}
\usepackage{float}
\usepackage{wrapfig}
\usepackage{multirow}
\usepackage{subcaption}
\usepackage[textsize=footnotesize,color=DarkOrange!40]{todonotes}
\usepackage[indLines=false,noEnd=false]{algpseudocodex}
\usepackage[nameinlink,capitalise]{cleveref}
\usepackage{algorithm}
\usepackage{enumitem}
\usepackage{cancel}
\usepackage{siunitx}
\usepackage{mdframed}
\usepackage{booktabs}

\mdfdefinestyle{examplebg}{
    backgroundcolor=black!6,%
    hidealllines=true,%
    innertopmargin=-.5em,%
    innerbottommargin=.5em,%
    innerleftmargin=.5em,%
    innerrightmargin=.5em,%
    footnoteinside=false,
    skipabove=2pt,
    skipbelow=2pt,
}
\surroundwithmdframed[style=examplebg]{example}

\colorlet{LightPlum}{Plum!30!white}

\numberwithin{algorithm}{section}
\makeatletter
\AddToHook{env/algorithmic/before}{\def\@currentcounter{ALG@line}}
\crefalias{ALG@line}{line}
\makeatother

\usepackage{tikz,pgfplots,pgfplotstable}
\usepgfplotslibrary{colorbrewer,fillbetween}

\pgfplotsset{compat=1.18}
\definecolor{c1}{RGB}{102,194,165}
\definecolor{c2}{RGB}{252,141,98}
\definecolor{c3}{RGB}{141,160,203}
\pgfplotsset{
    opt1/.style  ={color=Set2-A,  line width=1pt, solid},
    opt2/.style  ={color=Set2-B,  line width=1pt, dashed},
    opt3/.style  ={color=Set2-C,  line width=1pt, densely dotted},
    legend style={
        font=\scriptsize,
        draw=none,
        fill=none,
        at={(0.98,0.98)},
        anchor= north east,
        legend columns=1,
        column sep=5pt,
        legend cell align=left,
        inner sep=1pt,
    },
}

\newcommand{\proxold}[2]{\prox_{#1}(#2)}
\DeclareMathOperator{\prox}{prox}

\def\grad{\nabla}
\def\norm#1{\|#1\|}
\def\adaptnorm#1{\left\|#1\right\|}
\def\linear{\mathbb{L}}
\def\term#1{\emph{#1}}
\def\iprod#1#2{\langle #1, #2 \rangle}

\def\defeq{:=}
\def\N{\mathbb{N}}
\def\R{\mathbb{R}}
\def\C{\mathbb{C}}
\def\gap{\mathscr{G}}

\def\lagrangian{\mathscr{L}}
\def\triggerix{\mathscr{I}}
\def\extR{\widebar{\R}}
\DeclareMathOperator{\Dom}{dom}

\DeclareMathOperator{\Sym}{Sym}
\DeclareMathOperator{\Id}{Id}
\let\Re\relax\DeclareMathOperator{\Re}{Re}

\newcommand{\freevar}{\,\boldsymbol\cdot\,}

\def\polar#1{#1^\circ}
\def\bipolar#1{#1^{\circ\circ}}

\def\opt#1{\hat #1}
\def\optx{\opt x}
\def\opty{\opt y}
\def\optu{\opt u}

\def\this#1{#1^k}
\def\nexxt#1{#1^{k+1}}

\def\thisx{\this x}

\def\thisu{\this u}
\def\nextx{\nexxt x}

\let\subdiff\partial

\DeclareRobustCommand{\upto}{{{\mathchoice%
            {\rotatebox[origin=c]{20}{$\to$}}%
            {\rotatebox[origin=c]{20}{$\to$}}%
            {\rotatebox[origin=c]{20}{\scalebox{0.75}{$\to$}}}%
            {\rotatebox[origin=c]{20}{\scalebox{0.6}{$\to$}}}%
}}}
\def\uintermed#1{u^{#1}_+}
\def\xintermed#1{x^{#1}_+}
\def\yintermed#1{y^{#1}_+}
\def\Gbarcoarse#1{\bar{G}_H^{#1}}
\def\dualv#1{v^{k,#1}}

\manuscriptcopyright{}
\manuscriptlicense{}

\makeatletter
\def\hlinewd#1{%
    \noalign{\ifnum0=`}\fi\hrule \@height #1
    \futurelet\reserved@a\@xhline}
\makeatother

\title{Primal-dual multigrid methods for nonsmooth optimization}

\author{%
    Felipe Guerra\thanks{MODEMAT Research Center in Mathematical Modeling and Optimization, Quito, Ecuador \emph{and} Department of Mathematics, Escuela Politécnica Nacional (EPN), Quito, Ecuador. \email{edison.guerra@epn.edu.ec}}
    \and
    Tuomo Valkonen\thanks{MODEMAT \emph{and} EPN \emph{and} Department of Mathematics and Statistics, University of Helsinki, Finland. \email{tuomo.valkonen@iki.fi}, \orcid{0000-0001-6683-3572}}
}

\date{2026-08-05}

\begin{document}

\maketitle

\begin{abstract}
    In optimization, one often encounters problems of the form $\min_x F(x)+E(x)+G(Kx)$.
    In this work, we combine primal-dual algorithms with multigrid techniques for their solution.
    To link the the fine-grid and coarse-grid problems problems, we introduce a nonsmooth primal-dual coherence condition, and an efficient partially linearized line search procedure.
    Our work is motivated by total variation regularized inverse imaging problems, on which we demonstrate the efficacy of the method, being able to solve problems not previously possible with forward-backward multigrid methods.
\end{abstract}

\section{Introduction}
\label{sec:introduction}

Several first-order methods have been developed to solve nonsmooth optimization problems to which basic forward-backward splitting is not easily applicable.
These include the Primal-Dual Splitting Method (PDPS) \cite{pock2009algorithm} and the Alternating Direction Method of Multipliers (ADMM) \cite{gabay1983chapter,valkonen2017preconditioned}.
However, the computational cost of these methods can be high on large scale problems.
A promising way to mitigate the cost is to use multigrid (or multilevel) techniques.

The main idea behind multigrid techniques is to reduce the dimension of the original problem -- generally considered to be set in a finite-dimensional space -- by passing to a different but related problem in a lower-dimensional space. This is commonly referred to as the \emph{coarse problem}.
By solving the coarse problem, a direction of improvement is obtained for the original \emph{fine problem}.

Multigrid methods were popularized in \cite{brandt1977multilevel} for the efficient solution of large-scale linear systems arising from the discretization of elliptic PDEs; see \cite{briggs2000multigrid} for an introduction.
Subsequently, it was extended to smooth optimization in \cite{nash2000multigrid}.
Recently, various extensions of the approach have been proposed for nonsmooth optimization problems.
These recent works either combine the forward-backward method with multigrid \cite{ang2024mgprox, guerra2025multigrid}, or smoothen the optimization problem \cite{parpas2017multilevel}, which implies that the optimization problem must have a prox-simple structure.

\emph{To bypass restrictions faced by forward-backward techniques, we propose a nonsmooth Primal-Dual Multigrid with Coarse Correction method (PDMCC). It links the PDPS applied to both the coarse and fine problems via a nonsmooth primal-dual coherence condition.}
This is \cref{alg:PDMCC}.
We recall that the basic PDPS is sometimes called the Chambolle–Pock method \cite{chambolle2011first}, and the variant with an additional forward step the Condat–Vũ method \cite{condat2013primaldual,vu2013splitting}.
We review such algorithms in \cref{sec:pdps} through the preconditioned proximal point approach of \cite{he2014convergence}, subsequently employed in \cite{tuomov-proxtest,clason2020introduction}.

\begin{algorithm}[t]
    \caption{Primal-Dual Multigrid with Coarse Correction (PDMCC)}
    \label{alg:PDMCC}
    \begin{algorithmic}[1]
        \Require Functions $F,E,G$ and step length parameters $\tau,\sigma>0$ satisfying \cref{assumption:basic:fine:structure}.
        Line search model parameters $\epsilon_k,\rho_k>0$, ($k \in \N$).
        Line search parameter $\kappa \in (0, 1)$.
        A trigger condition.
        \State Choose an initial iterate $u^0\defeq (x^0,y^0) \in X\times Y$.
        \ForAll{$k \in \N$ until a chosen stopping criterion is fulfilled}
        \State
        Perform the fine-grid PDPS update $\uintermed{k} \defeq (\xintermed{k},\yintermed{k})$ as\label{line:standar:PDPS}
        \State $\xintermed{k}
            \defeq
            \proxold{\tau F}
            {x^k-\tau\bigl(\grad E(x^k)+K^*y^k\bigr)}$ and
        \Comment{Primal fine update}
        \State $\yintermed{k}
            \defeq
            \proxold{\sigma G^*}
            {y^k+\sigma K(2x^{k+1}-x^k)}.$ \Comment{Dual fine update}
        \If{the trigger condition is satisfied}\label{line:PDMCC:trigger}
        \State Choose coarse functions $F_H^k, E_H, (G_H^k)^*$ subject to \cref{assumption:coarse:basic,assumption:coherence:condition}.\label{line:coarse:functions}
        \State Form the descent direction $d^k$ with \cref{alg:coarse}.
        \Comment{Coarse correction}
        \State Find a line search step length $\theta \ge 0$ satisfying \eqref{eq:new:line:search} or \eqref{eq:new:line:search:surrogate}
        \State Perform the fine-grid update $u^{k+1}\defeq \uintermed{k} + \theta_k d^k$. \label{line:ls:update}
        \Else
        \State $u^{k+1} \defeq \uintermed{k}$. \label{line:PDPS:update}
        \EndIf
        \EndFor
    \end{algorithmic}
\end{algorithm}

\Cref{alg:PDMCC} applies to problems of the form
\begin{equation}
    \label{eq:general:optimization:problem}
    \min_{x \in X} F(x) + E(x) + G(Kx),
\end{equation}
where $F$ and $G$ are convex, possibly nonsmooth functions, $E$ is convex and smooth with $L$-Lipschitz gradient, and $K\in \linear(X; Y)$ on Hilbert spaces $X$ and $Y$.
We treat in \cref{sec:structure} the nonsmooth primal-dual coherence condition that lays out rules for designing $F_H$, $E_H$, $G_H$ and $K_H$ for a corresponding coarse problem.
When a trigger condition -- freely chosen by the user of the method -- is satisfied, $m$ coarse PDPS iterations are performed on these functions, starting from the restriction $v^{k,0} \defeq I_h^H \uintermed{k}$ of the fine primal-dual pre-iterate $\uintermed{k}=(\xintermed{k}, \yintermed{k})$, to yield a a direction $\this d = I_H^h(v^{k,m} - v^{k,0})$ from the prolonged initial coarse iterate to the prolonged final coarse iterate.

Integrating this direction into a primal-dual method presents significant challenges, because there is no obvious objective function for which to seek decrease: the Fenchel–Rockafellar duality gap can rarely be used in convergence proofs, while the more commonly employed Lagrangian duality gap depends on a base point, which is not the same on the two grids.
We show in \cref{sec:coarse-PDPS} that $\this d$ is a descent direction for
\begin{equation}
    \label{eq:intro:phik}
    \Phi^k(x, y) \defeq F(x) + E(x) + G^*(y) + \iprod{\yintermed{k}}{Kx}  - \iprod{K\xintermed{k}}{y},
\end{equation}
when an ergodic or non-ergodic construction is used for the coarse algorithm.
When the line search is amended with quadratic penalization, we obtain a bound on the Lagrangian gap of the fine problem, reminiscent of standard results for the PDPS.
As a result, the PDMCC preserves the convergence structure of the PDPS, modified only by a controllable error term.

Our method uses line search to incorporate coarse information into the fine algorithm.
In contrast to \cite{malitsky2018first}, where line search was first used with the PDPS -- without multigrid -- to adaptively adjust step length parameters without needing to know the norm of $K\in\mathbb{L}(X;Y)$, the line search we propose is more classical.
In particular, while in \cite{malitsky2018first} the primal and dual variables are updated using proximal steps, our method performs updates along a descent direction of $\Phi^k$.

The convergence of primal-dual algorithms can be studied through two key concepts: Fejér monotonicity of the generated sequence with respect to the set of optimal solutions, and ergodic convergence of the Lagrangian gap. Consequently, any fine-grid update performed through
line search must be carefully designed to ensure at least one of these properties—ideally, both.
In \cref{sec:line-search} we show that we can incorporate terms relevant for Fejér monotonicity into a line search procedure on $\Phi^k$.
This then allows proving for the PDMCC the ergodic convergence of the Lagrangian gap at the rate $O(1/N)$, as well as Fejér quasi-monotonicity.
The latter establishes weak convergence via Opial's lemma.
We also illustrate, in \cref{rem:nonconvex}, how our results can be extended to nonconvex $E$.

In \cref{sec:construction} we discuss ways to construct the coarse functions $E_H$, $G_H$, and $F_H$.
Finally, in \cref{sec:numerical}, we provide a series of \emph{numerical experiments} on total variation–regularized inverse imaging problems to validate the proposed method and highlight its computational advantages.

\section{Primal-Dual Proximal Splitting}
\label{sec:pdps}

We recall here the preconditioned proximal point approach \cite{he2014convergence,tuomov-proxtest,clason2020introduction} to the PDPS for \cref{eq:general:optimization:problem}. We follow \cite[Chapter~11]{clason2020introduction}.
To start, we recall basic notation. We write $\extR = [-\infty,+\infty]$.
For a convex function $F: X \to \extR$, $\partial F(x)$ denotes its subdifferential at $x$;
$F^*: X \to \extR$ the Fenchel conjugate; and $\prox_F$ the proximal operator.
If $K:X\to Y$ is linear and bounded, we write $K\in \linear(X; Y)$.

By the Fenchel--Rockafellar theorem, primal-dual solution pairs $\optu=(\optx, \opty) \in U \defeq X \times Y$ on the Hilbert spaces $X$ and $Y$, are characterized by
\begin{equation}
    \label{eq:PDPS:OC}
    -K^*\opty - \grad E(\optx) \in  \partial F(\optx)
    \quad\text{and}\quad
    K\optx \in \partial G^*(\opty).
\end{equation}
Based on this, the PDPS method with a forward step with respect to $E$ reads
\begin{equation}
    \label{eq:PDPS:update}
    \left\{
    \begin{aligned}
        x^{k+1} & \defeq \proxold{\tau F}{x^k - \tau[\grad E(x^k)+K^*y^k]},
        \\
        y^{k+1} & \defeq \proxold{\sigma G^*}{y^k+\sigma K(2x^{k+1}-x^k)}.
    \end{aligned}
    \right.
\end{equation}
In implicit form,
\begin{align*}
    0 & \in \partial F(x^{k+1}) + K^*y^{k+1} +\grad E(x^k) - K^*(y^{k+1} -y^k) + \tau^{-1}(x^{k+1}-x^k)
    \quad\text{and}
    \\
    0 & \in \partial G^*(y^{k+1}) - K x^{k+1} - K(x^{k+1} -x^k) + \sigma^{-1}(y^{k+1}-y^k).
\end{align*}
These inclusions take a convenient form in the product space $U=X \times Y$.
Let
\begin{equation}
    \label{eq:splitting:Gbar:and:Ebar}
    \bar{G}:U\to \extR,\quad
    \bar{G}(u) \defeq F(x) + G^*(y)\quad \text{and} \quad\bar{E}(u):U\to \extR,\quad \bar{E}(u)\defeq E(x).
\end{equation}
We also define $M, \Xi, \Lambda \in \linear(U; U)$ by
\begin{equation}
    \label{eq:def:M:Xi:Lambda}
    M\defeq
    \begin{pmatrix}
        \tau ^{-1}\Id & -K^*
        \\
        -K            & \sigma ^{-1}\Id
    \end{pmatrix}
    ,
    \quad
    \Xi \defeq
    \begin{pmatrix}
        0  & K^*
        \\
        -K & 0
    \end{pmatrix}
    ,
    \quad \text{and} \quad
    \Lambda \defeq
    \begin{pmatrix}
        L \Id & 0
        \\
        0     & 0
    \end{pmatrix}
    ,
\end{equation}
where $L\ge0$ is the Lipschitz constant of the gradient of $E$, and $\Id$ denotes the identity operator.
Then, the PDPS is characterized by the joint implicit inclusion
\begin{equation}
    \label{eq:PDPS:genFB}
    0 \in \partial \bar{G}(u^{k+1}) + \Xi u^{k+1} + \grad \bar{E}(u^k) + M(u^{k+1}-u^k).
\end{equation}
In other words, it is a preconditioned forward-backward method with the skew-adjoint perturbation $\Xi$, which does not arise as a differential.

To state basic results for the PDPS, we require:

\begin{assumption}
    \label{assumption:basic:fine:structure}
    $F,E:X\to \extR$ and $G:Y\to \extR$ are proper, convex and lower semicontinous with  $L$-Lipschitz $\grad E$, and  $K\in \linear(X; Y)$ on Hilbert spaces $X$ and $Y$.
    The steps length parameters $\tau,\sigma >0$ are chosen such that $\tau L+\tau \sigma\norm{K}_{\linear(X; Y)}^2 <1$.
\end{assumption}

\Cref{assumption:basic:fine:structure} implies \cite[Chapter 7]{clason2020introduction} that $\bar{E}$, defined in \cref{eq:splitting:Gbar:and:Ebar} satisfies for any $u,\optu,\tilde{u}\in U$ and $\Lambda \in \linear(U; U)$ from \cref{eq:def:M:Xi:Lambda} the \textit{three-point smoothness inequality}
\begin{equation}
    \label{eq:three-point-smoothness}
    \iprod{\grad \bar{E}(\optu)}{u-\tilde{u}} \ge \bar{E}(u)-\bar{E}(\tilde{u}) - \frac{1}{2}\norm{u-\optu}_\Lambda^2.
\end{equation}
Moreover, we have:

\begin{lemma}[{\cite[Lemma 9.12]{clason2020introduction}}]
    \label{lemma:positive:define:M}
    Let \cref{assumption:basic:fine:structure} hold.
    Then $M-\Lambda$ and $M$ are bounded and self-adjoint; $M-\Lambda$ is positive definite; and  $\norm{u}_M^2 \ge C\norm{u}^2$ for all $u \in X \times Y$, where $C \defeq (1- \sqrt{\tau \sigma}\norm{K}_{\linear(X; Y)})\min\{\tau^{-1},\sigma^{-1}\}>0$.
\end{lemma}

As a positive (semi-)definite and self-adjoint operator, $M$ defines the (semi)norm $\norm{u}_M \defeq \sqrt{\iprod{Mu}{u}}$ that satisfies for all $u,\optu,\tilde{u}\in U$ the \emph{three-point identity}
\begin{equation}
    \label{eq:preconditioned:three:point}
    \iprod{M(u-\optu)}{u-\tilde{u}} = \frac{1}{2}\norm{u-\optu}_M^2 - \frac{1}{2}\norm{\optu-\tilde{u}}_M^2 + \frac{1}{2}\norm{u-\tilde{u}}_M^2.
\end{equation}

To analyze convergence, we introduce the Lagrangian gap
\begin{gather}
    \label{eq:Lagrangian:gap}
    \gap_L(u;\optu)\defeq \Phi(u;\optu) - \Phi(\optu;\optu)
    \quad\text{for}\quad u,\optu\in U,
    \shortintertext{where\footnotemark}
    \label{eq:defi:funciton:Phi}
    \Phi:U\times U\to \extR,\qquad
    \Phi(u;\optu)\defeq \bar{G}(u) + \bar{E}(u) + \iprod{\Xi \optu}{u}.
\end{gather}
If $F$, $G$, and $E$ are convex, and $\optu$ solves the primal-dual optimality conditions \eqref{eq:PDPS:OC}, then the Lagrangian gap is non-negative. It is zero if $u=\optu$.
We have $\gap_L(u;\uintermed{k}) = \Phi^k(u) - \Phi(\uintermed{k}; \uintermed{k})$ for $\Phi^k$ of \eqref{eq:intro:phik}, which suggests why $d^k$ being a descent direction of $\Phi^k$ can be useful. The basic PDPS admits the following fundamental bound:

\footnotetext{$\Phi$ is not the Lagrangian $\lagrangian(x,y) = [F+E](x)+\iprod{Kx}{y}-G^*(y)$ that has $\gap_L(u; \hat u)=\lagrangian(x, \hat y)-\lagrangian(\hat x, y)$.}

\begin{theorem}[{\cite[Theorem 11.7]{clason2020introduction}}]
    \label{thm:basic:PDPS-inequality}
    Let \cref{assumption:basic:fine:structure} hold. Then $\{u^k=(x^k,y^k)\}_{k\in \N}$ generated by the PDPS \cref{eq:PDPS:update} satisfies for all $\optu\in U$ and $k\ge 0$ the bound
    \[
        \gap_L(u^{k+1};\optu) + \frac{1}{2}\norm{u^{k+1}-\optu}_M^2 + \frac{1}{2}\norm{u^{k+1}-u^k}_{M-\Lambda}^2
        \le \frac{1}{2}\norm{u^k-\optu}_M^2.
    \]
\end{theorem}

As $M \ge \Lambda$ by \cref{assumption:basic:fine:structure,lemma:positive:define:M}, summing this bound over $k=0,\ldots,N-1$, and letting $N \upto \infty$, we get ergodic $O(1/N)$ convergence of the gap \cite{clason2020introduction}.

\section{Coarse problem}
\label{sec:structure}

We now construct our overall approach to coarse approximations of \cref{eq:general:optimization:problem}.
The fundamental idea of our approach traces back to the smooth \term{coherence condition} \cite{nash2000multigrid}, explicitly expressed as such in \cite{parpas2017multilevel}.
A nonsmooth variant was introduced for forward-backward type methods in \cite{guerra2025multigrid}.
We recall the concept in \ref{sec:structure:motivation}, and then extend it to the primal-dual setting in \cref{sec:structure:coherence} after first defining basic notation and multigrid transfer operators in \cref{sec:structure:definitions}.

\subsection{Basic definitions; multigrid transfer operators}
\label{sec:structure:definitions}

We write $X_H$ and $Y_H$ for the coarse spaces corresponding to the primal space $X$ and the dual space $Y$.
All the spaces are assumed to be Hilbert spaces. Usually they are finite-dimensional, and satisfy $\dim X_H<\dim X$ and $\dim Y_H<\dim Y$.
We write
\[
    \left\{\begin{array}{l}
    \text{$u=(x,y) \in U \defeq X \times Y$ for the fine (primal, dual) variables; and}
    \\
    \text{$v = (\zeta, \xi) \in U_H \defeq X_H \times Y_H$ for the coarse (primal, dual) variables.}
    \end{array}\right.
\]

Transfer operators are the fundamental tool of multigrid methods: restriction from a fine grid to a coarse grid, and interpolation (or prolongation) from the coarse grid to the fine grid \cite{briggs2000multigrid}. Typically one is the adjoint of the other.
Since the PDMCC involves primal and dual variables, it is necessary to distinguish transfer operators associated with each of these spaces:
$P_h^H \in \linear(X; X_H)$ and $D_h^H \in \linear(Y; Y_H)$ are the primal and dual restriction operators, respectively.
We define the combined restriction operator $I_h^H \in \linear(U; U_H)$ by $I_h^H(x,y) \defeq (P_h^Hx, D_h^Hy)$.
Then the primal and dual prolongation operators are $P_H^h \defeq (P_h^H)^* \in \linear(X_H; X)$ and $D_H^h \defeq (D_h^H)^* \in \linear(Y_H; H)$.
The combined prolongation operator $I_H^h \in \linear(U_H; H)$ is $I_H^h \defeq (I_h^H)^*$, i.e., $I_H^h(\zeta, \xi) = (P_H^h\zeta, D_H^h\xi)$.

\subsection{Motivation for the coherence condition}
\label{sec:structure:motivation}

Consider the simple smooth problem $\min_x E(x)$.
Suppose that on iteration $k$, at the point $\thisx$, we want to pass to the coarse grid.
To do so, we have to construct a coarse objective $E_H^k$.
In \cite{nash2000multigrid,parpas2017multilevel}, this construction has to satisfy the \term{coherence condition}
\begin{equation}
    \label{eq:coarse:smooth-coherence}
    P_h^H \grad E(x^k) = \grad E_H^k(\zeta^{k,0})
    \quad\text{where}\quad
    \zeta^{k,0} \defeq P_h^H x^k.
\end{equation}
Then, any descent direction $d_H$ for $E_H^k$ at $\zeta^{k,0}$, i.e., $\iprod{E_h^k(\zeta^{k,0})}{d_H}<0$, can easily be translated into a descent direction for $E$.
Indeed,
\[
    \iprod{\grad E(x^k)}{P_H^h d_H}
    =
    \iprod{P_h^H\grad E(x^k)}{d_H}
    =
    \iprod{E_h^k(\zeta^{k,0})}{d_H}<0.
\]
One way to construct $E_H^k$ satisfying the smooth coherence condition \eqref{eq:coarse:smooth-coherence}, is to take \emph{any} differentiable $E_H: X_H \to \R$, and set
\begin{equation}
    \label{eq:coarse:ehk}
    E_H^k(\zeta)
    \defeq
    E_H(\zeta) + \iprod{r_H^k}{\zeta}
    \quad\text{for}\quad
    r_H^k \defeq P_h^H \grad E(x^k) - \grad E_H(P_h^H x^k).
\end{equation}
Intuitively, $E_H$ should be a “coarse version” of $E$, whereas $E_H^k$ corrects it for the restriction error.

Consider then the problem $\min_x J(x) \defeq F(x) + E(x)$, where $E$ is differentiable but $F$ is not.
We can then extend \eqref{eq:coarse:smooth-coherence} into the \term{nonsmooth coherence condition}
\begin{equation}
    \label{eq:coarse:nonsmooth-coherence}
    P_h^H \subdiff J(x^k) \subset \subdiff J_H^k(\zeta^{k,0}),
\end{equation}
This can be achieved as $J_H^k=E_H^k+F_H^k$, where $E_H^k$ is as in \eqref{eq:coarse:ehk}, and $F_H^k$ satisfies the nonsmooth coherence condition with respect to $F$.
The function $F_H^k$ can be constructed as an indicator function of a cone \cite{guerra2025multigrid}.
Again, a descent direction for $J_H^k$, will be a descent direction for $J$ \cite{guerra2025multigrid}.

\subsection{Nonsmooth primal-dual coherence condition}
\label{sec:structure:coherence}

A challenge with extending the coherence condition to primal-dual methods is that they are not necessarily monotone with respect to the Lagrangian, or the Fenchel–Rockafellar gap: they may not yield descent directions. It is, also, difficult to transfers of coarse gap to the fine gap, due to the bilinear terms in the gap, in other words, the skew-symmetric term $\Xi$ in the generalized forward-backward formulation \eqref{eq:PDPS:genFB}.

Nevertheless, motivated by the descent estimate of \cref{thm:basic:PDPS-inequality}, we construct the coarse problem such that we decrease the function (recall \cref{eq:defi:funciton:Phi,eq:intro:phik})
\begin{equation}
    \label{eq:phik}
    \Phi^k(u) \defeq \Phi(u; \uintermed{k}) = \bar{G}(u)+\bar{E}(u)+\iprod {\Xi \uintermed{k}}{u}.
\end{equation}
In the first “ergodic” variant of our coarse method, detailed in \cref{sec:coarse-PDPS}, we pick a primal-dual \emph{tilt vector} $w_H^k=(r_H^k, o_H^k) \in U_H$, associated with the smooth part $\bar{E}+\iprod{\Xi \uintermed{k}}{\freevar}$ of $\Phi^k$.
We then consider the coarse problem
\begin{equation}
    \label{eq:min:max:coarse:problem}
    \min_{\zeta \in X_H} \max_{\xi \in Y_H}~
    F_H^k(\zeta) + E_H(\zeta)
    + \iprod{\xi}{K_H \zeta}_{Y_H} - (G_H^k)^*(\xi)+ \iprod{r_H^k}{\zeta}_{X_H}
    - \iprod{o_H^k}{\xi}_{Y_H}.
\end{equation}
In the second “non-ergodic” variant of the method, the tilt vectors depend on the coarse iteration index $j$, i.e., $w_H^{k,j}=(r_H^{k,j}, o_H^{k,j})$.
In this case, the interpretation of minimizing \eqref{eq:min:max:coarse:problem} cannot be directly given, although the algorithm will be analogous.

The coarse functions $F_H^k$ and $(G_H^k)^*$ will need to satisfy the following assumptions.

\begin{assumption}[Basic coarse structure]
    \label{assumption:coarse:basic}
    $E_H, F_H^k : X_H \to \R$ and $G_H^k : Y_H \to \extR$ are proper, convex, and lower semicontinuous with $L_H$-Lipschitz $\grad E_H$, and $K_H \in \linear(X_H; Y_H)$ on Hilbert spaces $X_H$ and $Y_H$.
    The steps length parameters $\tau_H,\sigma_H >0$ are chosen such that $\tau_H L_H + \tau_H \sigma_H \norm{K_H}^2_{\mathbb{L}(X_H;Y_H)} <1$.
\end{assumption}

As in \cref{sec:pdps}, we introduce the extended coarse functions
$\Gbarcoarse{k}:U_H\to \extR$ and $\bar{E}_H(v):U_H\to \extR$
on the product space $U_H$, defined by
\begin{equation}
    \label{eq:splitting:Gbar:and:Ebar:coarse}
    \Gbarcoarse{k}(v) \defeq F_H^k(\zeta) + (G_H^k)^*(\xi)
    \quad \text{and}\quad
    \bar{E}_H(v)\defeq E_H(\zeta).
\end{equation}
The operators $M_H, \Xi_H,\Lambda_H \in \mathbb{L}(U_H;U_H)$ are defined by
\begin{equation}
    \label{eq:def:M:Xi:Lambda:coarse}
    M_H\defeq
    \begin{pmatrix}
        \tau_H ^{-1}\Id & -K_H^*
        \\
        -K_H            & \sigma_H ^{-1}\Id
    \end{pmatrix}
    ,
    \quad \Xi_H \defeq
    \begin{pmatrix}
        0    & K_H^*
        \\
        -K_H & 0
    \end{pmatrix}
    \quad \text{and}
    \quad
    \Lambda_H \defeq
    \begin{pmatrix}
        L_H \Id & 0
        \\
        0       & 0
    \end{pmatrix}
    ,
\end{equation}
where $L_H\ge0$ is the Lipschitz constant of the gradient of $E_H$.

We can now state our primal-dual coherence condition as a variant of  \cref{eq:coarse:nonsmooth-coherence}:

\begin{assumption}[Nonsmooth Primal-Dual Coherence Condition]
    \label{assumption:coherence:condition}
    For a given $k \in \N$, the fine-grid pre-iterate $\uintermed{k} \in U$, and the corresponding initial coarse iterate $v^{k,0} \in U_H$ (typically $I_h^H \uintermed{k}$) satisfy
    \[
        I_h^H \partial \bar{G}(\uintermed{k}) \subset \partial \Gbarcoarse{k}(v^{k,0}).
    \]
\end{assumption}

This condition can be decomposed into its \emph{primal} and \emph{dual} components
\[
    P_h^H \partial F(\xintermed{k})  \subset \partial F_H^k(\zeta^{k,0})
    \quad\text{and}\quad
    D_h^H \partial G^*(\yintermed{k}) \subset \partial (G_H^k)^*(\xi^{k,0}).
\]

\section{Coarse algorithm}
\label{sec:coarse-PDPS}

In \cref{sec:ergodic:variant,sec:non-ergodic:variant}, we present two coarse-grid primal-dual algorithms.
The first variant produces an ergodic descent direction, i.e., one that takes the average over the iterations.
The second avoids this averaging, and potentially expensive operator evaluations.
In \cref{sec:coarse:estimate}, we derive coarse-grid descent estimates for both variants.
Then, in \cref{sec:coarse:fine-descent}, we transfer these estimates to the fine grid.

\subsection{Ergodic variant}
\label{sec:ergodic:variant}

In the ergodic variant of the coarse-grid method, we take the $r_H^k \in X_H$ and $o_H^k \in Y_H$ of the coarse problem \cref{eq:min:max:coarse:problem} as
\[
    r_H^k \defeq P_h^H\big(\grad E(\xintermed{k}) + K^* \yintermed{k}\big)
    - (\grad E_H(\zeta^{k,0}) + K_H^* \xi^{k,0})
    \quad\text{and}\quad
    o_H^k\defeq - D_h^H K \xintermed{k} + K_H \zeta^{k,0},
\]
that is,
\[
    w_H^k \defeq (r_H^k, o_H^k) =  I_h^H (\grad \bar{E}(\uintermed{k})+ \Xi \uintermed{k}) - (\grad \bar{E}_H (v^{k,0}) + \Xi_H v^{k,0}).
\]
On each outer iteration $k \in \N$, denoting the coarse iteration number by $j$, and the coarse iterates by $v^{k,j}=(\zeta^{k,j}, \xi^{k,j})$, the PDPS \cref{eq:PDPS:update} for the general coarse problem \cref{eq:min:max:coarse:problem} expands as \cref{line:coarse:primal,line:coarse:dual} of \cref{alg:coarse} with $\vartheta = 1$.
This is valid for any $w_H^k$. Using the definitions \cref{eq:splitting:Gbar:and:Ebar:coarse,eq:def:M:Xi:Lambda:coarse}, we write the method in the implicit form
\begin{equation}
    \label{eq:ergodic:implicit}
    0 \in \partial \Gbarcoarse{k}(v^{k,j+1}) + \grad \bar{E}_H(v^{k,j}) + \Xi_H v^{k,j+1} + w_H^k + M_H(v^{k,j+1}-v^{k,j}).
\end{equation}

\begin{algorithm}[t]
    \caption{Coarse primal-dual method}
    \label{alg:coarse}
    \begin{algorithmic}[1]
        \Require Coarse functions $F_H^k$, $G_H^k$, and $E_H$ satisfying
        \cref{assumption:coarse:basic,assumption:coherence:condition}.
        Step length parameters $\tau_H,\sigma_H>0$.
        Iteration count $m\in\N^+$.
        Transfer operators $P_h^H$ and $D_h^H$.
        Choice of $\vartheta \in \{0,1\}$ (non-ergodic vs.~ergodic variant).
        \State $v^{k,0}\defeq (P_h^H \xintermed{k},D_h^H \yintermed{k}) \in U_H$.
        \State $a_H^k \defeq P_h^H\big(\grad E(\xintermed{k}) + K^* \yintermed{k}\big)
            - \grad E_H(\zeta^{k,0})$ and $b_H^k\defeq - D_h^H K \xintermed{k} $\label{line:tilt:ab}.
        \ForAll{$j=0,1,\ldots,m-1$}
        \State $\zeta^{k,j+1} \defeq \proxold{\tau_H F_H^k}{\zeta^{k,j} - \tau_H\big[\grad E_H(\zeta^{k,j}) + \vartheta K_H^*(\xi^{k,j}-\xi^{k,0}) + a_H^k\big]}$\label{line:coarse:primal}
        \Comment{Primal update}
        \State $\xi^{k,j+1} \defeq \proxold{\sigma_H (G_H^k)^*}{\xi^{k,j} + \sigma_H \big[2K_H(\zeta^{k,j+1}-\zeta^{k,j}) + \vartheta K_H(\zeta^{k,j}-\zeta^{k,0}) - b_H^k\big]}$\label{line:coarse:dual}
        \Comment{Dual update}
        \EndFor
        \State \Return $d = \tilde v^{k,j+1}-v^{k,0}$, where $\tilde v^{k,m}$ is defined by \cref{eq:tilde:v} and $v^{k,j}=(\zeta^{k,j},\xi^{k,j})$. \Comment{Descent direction}
    \end{algorithmic}
\end{algorithm}

\subsection{Non-ergodic variant}
\label{sec:non-ergodic:variant}

The solution of \eqref{eq:ergodic:implicit}, i.e., \cref{alg:coarse} with $\vartheta =1$, requires the evaluation of $K_H\in \mathbb{L}(X_H;Y_H)$ as well as its adjoint on each iteration $j$.
This can be computationally expensive and even lead to numerical instabilities.
To ameliorate these issues, we now make the tilt vectors dependent on the coarse iteration $j$, using in place of $r_H^k$ and $o_H^k$ the vectors
\[
    r_H^{k,j} \defeq P_h^H\big(\grad E(\xintermed{k}) + K^* \yintermed{k}\big)
    - (\grad E_H(\zeta^{k,0}) + K_H^* \xi^{k,j })
    \quad\text{and}\quad
    o_H^{k,j}\defeq - D_h^H K \xintermed{k} + K_H \zeta^{k,j},
\]
that is, in place of $w_H^k$, the vector
\begin{equation}
    \label{eq:non-ergodic:def:whk}
    w_H^{k,j} \defeq I_h^H (\grad \bar{E}(\uintermed{k})+ \Xi \uintermed{k}) - (\grad \bar{E}_H (v^{k,0}) + \Xi_H v^{k,j}).
\end{equation}
We correspondingly modify the implicit algorithm \eqref{eq:ergodic:implicit} into
\begin{equation}
    \label{eq:nonergodic:implicit}
    0 \in \partial \Gbarcoarse{k}(v^{k,j+1}) + \grad\bar{E}_H(v^{k,j}) +
    \underbrace{\Xi_H v^{k,j+1} + w_H^{k,j} + M_H(v^{k,j+1}-v^{k,j}).}_{(\Xi_H + M_H)(v^{k,j+1}-v^{k,j}) + I_h^H(\grad\bar E(\uintermed{k}) + \Xi \uintermed{k})  - \grad \bar E(v^{k,0})}
\end{equation}
Using the rearrangement under the brace, which follows \eqref{eq:non-ergodic:def:whk} and the definition of the operators $M_H$ and $\Xi_H$ in \cref{eq:def:M:Xi:Lambda:coarse}, in explicit form, we see that \eqref{eq:nonergodic:implicit} reads as \cref{line:coarse:primal,line:coarse:dual} of \cref{alg:coarse} with $\vartheta =0$.

\subsection{Descent}
\label{sec:coarse:estimate}

We can write both \cref{eq:ergodic:implicit,eq:nonergodic:implicit}  as
\begin{subequations}
    \label{eq:implicit:equation:common}
    \begin{gather}
        \label{eq:implicit:equation:common:implicit}
        0 \in \partial \Gbarcoarse{k}(\dualv{j+1}) + \Xi_H(\dualv{j+1}-v_\vartheta^j) +\grad \bar E(\dualv{j}) + s_H^k
        + M_H(\dualv{j+1}-\dualv{j}),
        \intertext{where $v_\vartheta^j \defeq (1-\vartheta)\dualv{j} + \vartheta\dualv{0}$,  $\vartheta \in \{0,1\}$ indicates the variant, and}
        \label{eq:sHk:common:ergodic:non-ergodic}
        s_H^k = (a_H^k, b_H^k) \defeq I_h^H (\grad \bar{E}(\uintermed{k})+ \Xi \uintermed{k}) - \grad \bar{E}_H (v^{k,0}) \in U_H.
    \end{gather}
\end{subequations}
With this, we obtain a tilted descent estimate in the coarse grid:

\begin{theorem}
    \label{theorem:coarse:estimate}
    Let \cref{assumption:coarse:basic} hold.
    Then, for any initial coarse point $v^{k,0}\in U_H$, and $v^{k,1},\ldots,v^{k,m}$ generated through \eqref{eq:implicit:equation:common}, we have
    \begin{gather}
        \label{eq:coarse:estimate:result}
        [\Gbarcoarse{k} + \bar{E}_H](\tilde v^{k,m}) + \iprod{s_H^k}{\tilde v^{k,m}-v^{k,0}} + Q(v^{k,0},\ldots,v^{k,m},m) \le [\Gbarcoarse{k} + \bar{E}_H](v^{k,0}).
        \shortintertext{where}
        \nonumber
        Q(v^{k,0},\ldots,v^{k,m},m)\defeq \frac{1}{2m} \left[\norm{\dualv{m}-\dualv{0}}_{M_H}^2 + \sum_{j=0}^{m-1}\norm{\dualv{j+1}-\dualv{j}}_{M_H-\Lambda_H}^2\right]\ge 0
        \shortintertext{and}
        \label{eq:tilde:v}
        \tilde v^{k,m} \defeq
        \begin{cases*}
            \frac{1}{m}\sum_{j=0}^{m-1} \dualv{j+1}, & for the ergodic variant \eqref{eq:ergodic:implicit},
            \\
            \dualv{m},                               & for the non-ergodic variant \eqref{eq:nonergodic:implicit}.
        \end{cases*}
    \end{gather}
\end{theorem}

\begin{proof}
    We write $v^{j} \defeq v^{k,j}$ and $\bar J \defeq \Gbarcoarse{k} + \bar{E}_H$ for brevity.
    By the convexity of $\Gbarcoarse{k}$ and the three-point smoothness inequality on $\bar{E}$, we know that
    \[
        \iprod{\partial \Gbarcoarse{k}(v^{j+1}) + \grad \bar{E}_H(v^j)}{v^{j+1}-v_\vartheta^j} \ge \bar J(v^{j+1})
        - \bar J(v_\vartheta^j) - \frac{1}{2}\norm{v^{j+1}-v^j}_{\Lambda_H}^2.
    \]
    Applying $\iprod{\freevar}{v^{j+1}-v_\vartheta^j}$ on both sides of \cref{eq:implicit:equation:common:implicit} and using this inequality, yields
    \begin{multline*}
        \iprod{\Xi_H(v^{j+1}-v_\vartheta^j)}{v^{j+1}-v_\vartheta^j} + \iprod{s_H^k}{v^{j+1}-v_\vartheta^j}
        + \iprod{M_H(v^{j+1}-v^j)}{v^{j+1}-v_\vartheta^j}
        \\
        \le
        \bar J(v_\vartheta^j) - \bar J(v^{j+1}) + \frac{1}{2}\norm{v^{j+1}-v^j}_{\Lambda_H}^2.
    \end{multline*}
    Using the skew-adjointness of $\Xi_H$ and the \emph{three-point identity} \cref{eq:preconditioned:three:point}, this becomes
    \begin{equation}
        \label{eq:common:inequality}
        \bar J(v^{j+1}) +
        \frac{1}{2}\norm{v^{j+1}-v_\vartheta^j}_{M_H}^2
        + \iprod{s_H^k}{v^{j+1}-v_\vartheta^j}
        + \frac{1}{2}\norm{v^{j+1}-v^j}_{M_H-\Lambda _H}^2
        \le
        \bar J(v_\vartheta^j) + \frac{1}{2}\norm{v^j -v_\vartheta^j}_{M_H}^2.
    \end{equation}

    \textbf{Ergodic case:}
    When $\vartheta = 1$, we have $v_\vartheta^j = v^0$. Summing \cref{eq:common:inequality} over $j=0,\ldots,m-1$, thus, yields
    \[
        \sum _{j=0}^{m-1} \left[
            \bar J(v^{j+1})
            + \iprod{s_H^k}{v^{j+1}-v^0}
            + \frac{1}{2}\norm{v^{j+1}-v^j}_{M_H-\Lambda_H}^2
            \right]
        + \frac{1}{2}\norm{v^m-v^0}_{M_H}^2
        \le
        m\bar J(v^0).
    \]
    According to \cref{assumption:coarse:basic}, $M_H$ is positive definite, so $\norm{\freevar}_{M_H}^2$ is a convex function.
    Thus, applying Jensen's inequality to $\bar J$, we obtain \cref{eq:coarse:estimate:result} for $\tilde v^{k,m} = \frac{1}{m}\sum _{j=0}^{m-1}v^{k,j+1}$.

    \textbf{Non-ergodic case:}
    When $\vartheta = 0$, we have $v_\vartheta^j=v^j$. Summing \cref{eq:common:inequality} over $j=0,\ldots,m-1$ now yields \cref{eq:coarse:estimate:result} in the form
    \[
        \bar J(v^m) + \iprod{s_H^k}{v^m-v^0}
        + \frac{1}{2}\sum_{j=0}^{m-1}\norm{v^{j+1}-v^j}_{M_H}^2
        + \frac{1}{2} \sum_{j=0}^{m-1}\norm{v^{j+1}-v^j}_{M_H-\Lambda_H}^2
        \le \bar J(v^0).
    \]
\end{proof}

To prove descent instead of mere non-increase, we use the next lemma.

\begin{lemma}
    \label{lemma:bound:of:Q}
    Let \cref{assumption:coarse:basic} hold, $m\ge1$, and define
    \[
        c_Q \defeq \frac{2(1-\sqrt{\tau_H\sigma_H(1-\tau_H L_H)^{-1}}\norm{K_H})\min\{(1-\tau_H L_H)\tau_H^{-1},\sigma_H^{-1}\}}{m(m+1)^2},
    \]
    If $v^{k,0}\in U_H$ does not solve \cref{eq:min:max:coarse:problem}, then, for $\tilde v^{k,m}$ defined by \cref{eq:tilde:v}, we have
    \[
        Q(v^{k,0},\ldots,v^{k,m},m)
        \ge
        c_Q \norm{\tilde v^{k,m}-v^{k,0}}^2.
    \]
\end{lemma}

\begin{proof}
    \textbf{Non-ergodic case:}
    By \cref{assumption:coarse:basic,lemma:positive:define:M}, $M_H-\Lambda_H$ is self-adjoint and positive definite; consequently, $\norm{\freevar}_{M_H-\Lambda_H}^2$ is convex and non-negative.
    We have $\tilde v^{k,m} = v^{k,m}$ and $\frac{1}{2m} \ge \frac{2}{m(m+1)^2}$, hence $(m+1)^2 \ge 4$ for all $m\ge1$.
    Therefore,%
    \begin{equation*}
        \begin{split}
            Q(v^{k,0},\ldots,v^{k,m},m)
             &
            =
            \frac{1}{2m} \left[\norm{\dualv{m}-\dualv{0}}_{M_H}^2 + \sum_{j=0}^{m-1}\norm{\dualv{j+1}-\dualv{j}}_{M_H-\Lambda_H}^2\right]
            \\
             &
            \ge \frac{1}{2m}\norm{\dualv{m}-\dualv{0}}_{M_H}^2
            \\
             &
            \ge
            \frac{(1-\sqrt{\tau_H\sigma_H}\norm{K_H})\min\{\tau_H^{-1},\sigma_H^{-1}\}}{2m}\norm{\tilde v^{k,m}-\dualv{0}}^2
            \\
             &
            \ge c_Q \norm{\tilde v^{k,m}-\dualv{0}}^2.
        \end{split}
    \end{equation*}

    \textbf{Ergodic case:}
    Recall that, now, $\tilde v^{k,m} = \frac{1}{m}\sum_{j=0}^{m-1}\dualv{j+1}$.
    We first prove that
    \[
        \tilde v^{k,m} -\dualv{0} = \sum_{j=0}^{m-1}\frac{m-j}{m}[\dualv{j+1}-\dualv{j}]
        \quad\text{for all}\quad m\ge1.
    \]
    We use induction.
    The base $m=1$ is clear. Assuming the expression for $m=n$, to prove it for $m=n+1$, we
    use the telescoping sum $\dualv{n+1}-\dualv{0} = \sum_{j=0}^n[\dualv{j+1}-\dualv{j}]$ to rearrange
    \[
        \tilde v^{k,n+1} -\dualv{0}
        =
        \frac{1}{n+1}\sum_{j=0}^{n-1}[\dualv{j+1}-\dualv{0}] + \frac{1}{n+1}[\dualv{n+1}-\dualv{0}]
        =
        \sum_{j=0}^n \frac{n+1-j}{n+1}[\dualv{j+1}-\dualv{j}].
    \]
    This completes the induction.

    Defining $\bar \lambda = \sum_{j=0}^{m-1}(m-j)$ and $\lambda_j = (m-j)/\bar \lambda$, we now obtain
        {\abovedisplayskip=5pt
            \[
                \norm{\tilde v^{k,m} - \dualv{0}}_{M_H-\Lambda_H}^2
                =
                \frac{\bar \lambda^2}{m^2} \adaptnorm{\sum_{j=0}^{m-1}\lambda _j[\dualv{j+1}-\dualv{j}]}_{M_H-\Lambda_H}^2.
            \]
        }
    Now, since $\sum_{j=0}^{m-1}\lambda_j = 1$ with $\lambda_j \in (0,1]$ for all $j$, Jensen's inequality gives
    \[
        \norm{\tilde v^{k,m} - \dualv{0}}_{M_H-\Lambda_H}^2 \le
        \frac{\bar \lambda^2}{m^2} \sum_{j=0}^{m-1} \lambda_j\norm{\dualv{j+1}-\dualv{j}}_{M_H-\Lambda_H}^2
        \le \frac{\bar \lambda^2}{m^2} \sum_{j=0}^{m-1}\norm{\dualv{j+1}-\dualv{j}}_{M_H-\Lambda_H}^2.
    \]
    Finally, by the last inequality, the definition of $Q(\dualv{0},\ldots,\dualv{m},m)$, and \cref{assumption:coarse:basic}, which ensures that $M\ge0$, we obtain
    \begin{equation*}
        \begin{split}
            Q(v^{k,0},\ldots,v^{k,m},m)
             &
            =
            \frac{1}{2m} \left[\norm{\dualv{m}-\dualv{0}}_{M_H}^2 + \sum_{j=0}^{m-1}\norm{\dualv{j+1}-\dualv{j}}_{M_H-\Lambda_H}^2\right]
            \\
             &
            \ge
            c_Q \norm{\tilde v^{k,m}-\dualv{0}}^2.
            \qedhere
        \end{split}
    \end{equation*}
\end{proof}

\subsection{Descent in the fine grid}
\label{sec:coarse:fine-descent}

We now transfer the descent estimate of \cref{theorem:coarse:estimate} to the fine grid.

\begin{corollary}
    \label{corollary:descent:direction}
    Let \cref{assumption:coarse:basic,assumption:coherence:condition} hold. Then for any initial coarse iterate $v^{k,0}\in U_H$, which does not solve the coarse problem \cref{eq:min:max:coarse:problem}, the descent direction $d^k \defeq  I_H^h(\tilde v^{k,m} -v^{k,0})\in  U$ generated through \eqref{eq:implicit:equation:common}, satisfies
    \[
        \norm{d^k} \le \tilde c_Q\inf_{g^k\in \partial \bar G(\uintermed{k})}\norm{g^k + \grad \bar E(\uintermed{k})+\Xi\uintermed{k}}
        \quad \text{and}\quad [\Phi^k]'(\uintermed{k};d^k)<0,
    \]
    where $\tilde v^{k,m}$ is defined by \cref{eq:tilde:v} and $\tilde c_Q \defeq \norm{I_H^h}/c_Q$.
\end{corollary}

\begin{proof}
    Let $s_H^k \in U_H$ be given by \cref{eq:sHk:common:ergodic:non-ergodic}.
    By its definition and the convexity of $\bar{E}_H$,
    \begin{equation}
        \label{eq:non-ergodic:whk:expand}
        \begin{split}
            \iprod{s_H^k}{\tilde v^{k,m} -v^{k.0}}
             &
            = \iprod{I_h^H(\grad \bar{E}(\uintermed{k}) + \Xi \uintermed{k})}{\tilde v^{k,m} -v^{k.0}} + \iprod{\grad \bar{E}_H(v^{k.0})}{v^{k.0}-\tilde v^{k,m}}
            \\
             &
            \ge \iprod{I_h^H(\grad \bar{E}(\uintermed{k}) + \Xi \uintermed{k})}{\tilde v^{k,m} -v^{k.0}} + \bar{E}_H (v^{k.0}) - \bar{E}_H(\tilde v^{k,m}).
        \end{split}
    \end{equation}
    We now recall the coarse descent estimate \cref{eq:coarse:estimate:result} obtained in \cref{theorem:coarse:estimate},
    \[
        \Gbarcoarse{k}(\tilde v^{k,m}) + \bar{E}_H(\tilde v^{k,m}) + \iprod{s_H^k}{\tilde v^{k,m}-v^{k,0}} +
        Q(v^{k,0},\ldots,v^{k,m},m)\le \Gbarcoarse{k}(v^{k,0}) + \bar{E}_H(v^{k,0}).
    \]
    Combining \cref{eq:non-ergodic:whk:expand} with the last inequality, yields
    \begin{equation}
        \label{eq:non-ergodic:combined:bound}
        \iprod{I_h^H(\grad \bar{E}(\uintermed{k}) + \Xi \uintermed{k})}{\tilde v^{k,m} -v^{k.0}}
        + \Gbarcoarse{k}(\tilde v^{k,m}) - \Gbarcoarse{k}(v^{k,0})
        + Q(v^{k,0},\ldots,v^{k,m},m)\le 0.
    \end{equation}
    Using the nonsmooth primal-dual coherence condition of \cref{assumption:coherence:condition}, we obtain
    \[
        \iprod{I_h^H(g^k + \grad \bar{E}(\uintermed{k}) + \Xi \uintermed{k})}{\tilde v^{k,m}-v^{k,0}}
        + Q(v^{k,0},\ldots,v^{k,m},m) \le 0
        \quad\text{for all}\quad g^k \in \partial\bar{G}(\uintermed{k}).
    \]
    Since $v^{k,0}$ does not solve the coarse problem, we have $v^{k,1} \ne v^{k,0}$.
    It follows that $Q(\dualv{0},\ldots,\dualv{m},m)>0$.
    Hence, using the definition of $d^k$, and taking the supremum over $g^k \in \partial\bar{G}(\uintermed{k})$, yields $[\Phi^k]'(\uintermed{k};d^k)< 0$.
    Moreover, by \cref{lemma:bound:of:Q}, it follows
    \[
        \iprod{I_h^H(g^k + \grad \bar{E}(\uintermed{k}) + \Xi \uintermed{k})}{\tilde v^{k,m}-v^{k,0}}
        + c_Q \norm{\tilde v^{k,m} - \dualv{0}}^2 \le 0
        \quad\text{for all}\quad g^k \in \partial\bar{G}(\uintermed{k}).
    \]
    Since $I_H^h \in \mathbb{L}(U_H;U)$, also using the Cauchy--Schwarz inequality we get
    \[
        \frac{c_Q}{\norm{I_H^h}^2} \norm{d^k}^2 \le \norm{d^k}\norm{g^k + \grad \bar{E}(\uintermed{k}) + \Xi \uintermed{k}}
        \quad\text{for all}\quad g^k \in \partial\bar{G}(\uintermed{k}).
    \]
    Taking the infimum over $g^k$, we get the claim.
\end{proof}

\begin{remark}[Necessity and rejection of coarse corrections]
    \Cref{corollary:descent:direction} requires the initial coarse iterate to not solve \cref{eq:min:max:coarse:problem}.
    If this condition is not satisfied, we have can take $d^k=0$ and $\theta=0$ in \cref{alg:PDMCC}, rejecting the coarse correction.
    In fact, the nonsmooth primal-dual coherence condition ensures that if $\uintermed{k} \in U$ is an optimal solution to the fine problem \cref{eq:general:optimization:problem}, then $v^{k,0} \in U_H$ is an optimal solution to the coarse problem; consequently, the coarse correction is unnecessary.
\end{remark}

\section{Line search for the coarse correction}
\label{sec:line-search}

In this section, we describe the second stage of coarse correction: an inexpensive line search that provides sufficient descent in a form that we can use in our main convergence proof in \cref{sec:convergence}.
We already know from these results that $d^k$ is a descent direction for $\Phi^k$ defined in \eqref{eq:phik}.
In \cref{sec:line-search:pdmcc}, we use this result to construct a line search procedure for the Lagrangian gap, compatible with the PDPS convergence estimate, \cref{thm:basic:PDPS-inequality}.
Then, in \cref{sec:line-search:linearised}, we explain how to reduce the computational cost of the line search by reusing the information used in the construction of the coarse problem.

\subsection{Basic procedure}
\label{sec:line-search:pdmcc}

We first prove the existence of an interval of line search parameters $\theta$ such that
\[
    \gap_L(\uintermed{k}+\theta d^k;\optu) + \frac{1}{2}\norm{\uintermed{k}+\theta d^k -\optu}_M^2
    \le
    \frac{1}{2}\norm{u^k-\optu}_M^2 + \frac{\varepsilon_k}{2}\norm{\uintermed{k}-\optu}_M^2 + \rho_k
\]
for chosen penalty parameters $\varepsilon_k,\rho_k>0$,
Choosing these parameters small, we can thus bound the Lagrangian gap arbitrarily well by the the squared distance of $\thisu$ to a minimiser $\optu$.
The additive penalty $\rho_k>0$ prevents the line search step length $\theta$ from becoming arbitrarily small.

To start, we recall the basic Armijo line search result for $\Phi^k$ \cite{nocedal2006numerical}:

\begin{lemma}
    \label{lemma:line:search}
    On the given fine iteration $k \in \N$, suppose \cref{assumption:coarse:basic,assumption:coherence:condition} hold for the initial coarse iterate $v^{k,0} = I_h^H \uintermed{k}$, which does not solve \cref{eq:min:max:coarse:problem}.
    Let $d^k$ be defined by \cref{corollary:descent:direction}.
    Then, for any $\kappa\in (0,1)$, there exist $\theta_0> 0$ such that
    \begin{equation}
        \label{eq:line:search:result}
        \Phi ^k(\uintermed{k}+\theta d^k) \le \Phi ^k(\uintermed{k}) + \kappa \theta [\Phi ^k]'(\uintermed{k};d^k)
        \quad\text{for all}\quad 0<\theta\le \theta _0.
    \end{equation}
\end{lemma}

\begin{proof}
    By \cref{corollary:descent:direction}, we know that $d^k$ is a descent direction for $\Phi^k$ at $\uintermed{k}$, that is, $[\Phi^k]'(\uintermed{k};d^k) < 0$. The rest follows from the definition of the directional derivative.
\end{proof}

The next lemma provides the basis for our basic line search procedure.
There, using the constant $C > 0$ be the from \cref{lemma:positive:define:M}, we define $W, Z \in \linear(U; U)$ as
\[
    Z \defeq
    \begin{pmatrix}
        \tau^{-1} \Id & -2K^*
        \\
        0             & \sigma ^{-1}\Id
    \end{pmatrix}
    \quad\text{and}\quad
    W = \frac{1}{C}Z Z^*.
\]

\begin{lemma}
    \label{lemma:improve:line:search}
    On the given fine iteration $k \in \N$, suppose \cref{assumption:coarse:basic,assumption:coherence:condition} hold for the initial coarse iterate $v^{k,0} = I_h^H \uintermed{k}$, which does not solve \cref{eq:min:max:coarse:problem}.
    Let $\varepsilon_k,\rho_k > 0$ and $\kappa \in (0, 1)$.
    Then there exists $\theta_0> 0$ such that
    \begin{equation}
        \label{eq:new:line:search}
        \Phi^k(\uintermed{k}+\theta d^k) + \frac{1}{2}\norm{\theta d^k}_M^2 + \frac{1}{2\varepsilon_k}\norm{\theta d^k}_W^2
        \le \Phi^k(\uintermed{k}) + \kappa \theta [\Phi ^k]'(\uintermed{k};d^k) + \frac{\rho_k}{2}
    \end{equation}
    for all $0<\theta\le \theta _0$ and $d^k$ defined in \cref{corollary:descent:direction}.
\end{lemma}

\begin{proof}
    We have $[\tilde{\Phi}^k]'(\uintermed{k}) = [\Phi ^k]'(\uintermed{k}) < 0$ for
    \[
        \tilde{\Phi}^k(u)\defeq \Phi ^k(u) + \frac{1}{2}\norm{u-\uintermed{k}}_M^2 + \frac{1}{2\varepsilon_k }\norm{u-\uintermed{k}}_W^2.
    \]
    \Cref{corollary:descent:direction} shows that $d^k = I_H^h(\tilde v^{k,m} -\dualv{0})$ is a descent direction for $\Phi^k$ at $\uintermed{k}$, hence also for $\tilde{\Phi}^k$.
    \Cref{lemma:line:search} applied to $\tilde{\Phi}^k$, thus, shows the existence of $\theta_0>0$ such that
    $
        \tilde{\Phi}^k(\uintermed{k}+\theta d^k) \le \tilde{\Phi}^k(\uintermed{k}) + \kappa\theta [\tilde{\Phi}^k]'(\uintermed{k},d^k)
    $
    for all $0<\theta\le \theta_0$.
    Now, applying the definition of $\tilde{\Phi}$ and $\rho_k/2>0$ then yields \eqref{eq:new:line:search}.
\end{proof}

The next result transforms some terms in the line search criterion.

\begin{lemma}
    \label{lemma:inner:product:inequality:line:search}
    Let $\varepsilon_k > 0$. Suppose that \cref{assumption:basic:fine:structure} holds. Then, for any $u,\optu,\tilde{u} \in U$, we have
    \[
        \iprod{\Xi (u -\optu)}{\tilde{u}} + \frac{1}{2} \norm{u -\optu}_M^2 \ge \frac{1}{2}\norm{u+\tilde{u}-\optu}_M^2
        - \frac{\varepsilon_k}{2}\norm{u -\optu}_M^2 -
        \frac{1}{2\varepsilon_k}\norm{\tilde{u}}_W^2 - \frac{1}{2}\norm{\tilde{u}}_M^2.
    \]
\end{lemma}

\begin{proof}
    The operator $M$ of \cref{eq:def:M:Xi:Lambda} can be split as $M = Z + \Xi$.
    With this, we obtain
    \begin{equation}
        \label{eq:inner:product:identity1}
        \iprod{\Xi(u-\optu)}{\tilde{u}} + \frac{1}{2}\norm{u-\optu}_M^2
        = \iprod{M(u-\optu)}{\tilde{u}} + \frac{1}{2}\norm{u-\optu}_M^2
        - \iprod{Z(u - \optu)}{\tilde{u}}.
    \end{equation}
    By the Pythagorean identity, we have
    \[
        \iprod{M(u-\optu)}{\tilde{u}} + \frac{1}{2}\norm{u -\optu}_M^2= \frac{1}{2} \big[\norm{u + \tilde{u} - \optu}_M^2 - \norm{\tilde{u}}_M^2 \big],
    \]
    and since $\norm{u}_W^2 = \norm{C^{-1/2}Z^*u}^2$, we have
    \[
        \begin{split}
            -\iprod{Z(u -\optu)}{\tilde{u}}
             &
            = - \iprod{u -\optu}{\theta Z^* d^k}
            \\
             &
            =
            \frac{1}{2}\big[\norm{(C\varepsilon_k)^{1/2}(u-\optu)-(C\varepsilon_k)^{-1/2}\theta Z^* d^k}^2 -
                           C\varepsilon_k\norm{u -\optu}^2 - \frac{1}{\varepsilon_k}\norm{\tilde{u}}_W^2\big].
        \end{split}
    \]
    Replacing the last two identities in \cref{eq:inner:product:identity1} yields
    \begin{equation*}
        \begin{split}
            \iprod{\Xi(u-\optu)}{\tilde{u}}
            + \frac{1}{2}\norm{u-\optu}_M^2
             &
            =
            \frac{1}{2}\Bigl[
                           \norm{u + \tilde{u} - \optu}_M^2
                           -C\varepsilon_k\norm{u -\optu}^2
                           - \frac{1}{\varepsilon_k}\norm{\tilde{u}}_W^2
                           \\
                           &
                           - \norm{\tilde{u}}_M^2
                           +
                           \norm{(C\varepsilon_k)^{1/2}(u-\optu)-(C\varepsilon_k)^{-1/2}\theta Z^* d^k}^2
                           \Bigr].
        \end{split}
    \end{equation*}
    Finally, we discard undesired non-negative terms and use $C\norm{u}^2 \le \norm{u}_M^2$.
\end{proof}

Finally, we have our desired result.

\begin{theorem}
    \label{theorem:lagrangian:gap:descent:convergence}
    On the given fine iteration $k \in \N$, suppose \cref{assumption:basic:fine:structure,assumption:coarse:basic,assumption:coherence:condition} hold for the initial coarse iterate $v^{k,0} = I_h^H \uintermed{k}$, which does not solve \cref{eq:min:max:coarse:problem}.
    Moreover, let $\varepsilon_k,\rho_k > 0$, and let $\theta>0$ (which exists, by \cref{lemma:improve:line:search}) satisfy \eqref{eq:new:line:search}.
    Then, for all $\optu \in U$ and  $d^k$ defined in \cref{corollary:descent:direction}, we have
    \[
        \gap_L(\uintermed{k}+\theta d^k;\optu) + \frac{1}{2}\norm{\uintermed{k}+\theta d^k -\optu}_M^2
        + \frac{1}{2}\norm{\uintermed{k}-u^k}_{M-\Lambda}^2\le
        \frac{1}{2}\norm{u^k-\optu}_M^2 + \frac{\varepsilon _k}{2}\norm{\uintermed{k}-\optu}_M^2 + \frac{\rho_k}{2}.
    \]
\end{theorem}

\begin{proof}
    By \cref{thm:basic:PDPS-inequality} and \cref{line:standar:PDPS} of \cref{alg:PDMCC}, the fine-grid pre-iterate $\uintermed{k}$ satisfies
    \[
        \gap_L(\uintermed{k};\optu) + \frac{1}{2} \norm{\uintermed{k}-\optu}_M^2 + \frac{1}{2}\norm{\uintermed{k}-u^k}_{M-\Lambda}^2\le \frac{1}{2}\norm{u^k - \optu}_M^2.
    \]
    Using \eqref{eq:Lagrangian:gap}, \eqref{eq:new:line:search} rewrites as
    \[
        \gap_L(\uintermed{k}+\theta d^k;\optu) + \iprod{\Xi(\uintermed{k} -\optu)}{ \theta d^k}
        + \frac{1}{2}\big[\norm{\theta d^k}_M^2 + \frac{1}{\varepsilon _k}\norm{\theta d^k}_W^2\big]
        \le \gap_L (\uintermed{k};\optu) + \frac{\rho_k}{2}.
    \]
    Combining the last last two inequalities yields
    \begin{equation*}
        \begin{split}
            \gap_L(\uintermed{k}+\theta d^k;\optu)
            +
             & \iprod{\Xi(\uintermed{k} -\optu)}{ \theta d^k}
            + \frac{1}{2} \norm{\uintermed{k}-\optu}_M^2
            \\
             &
            + \frac{1}{2}\norm{\uintermed{k}-u^k}_{M-\Lambda}^2
            + \frac{1}{2}\norm{\theta d^k}_M^2 + \frac{1}{2\varepsilon _k}\norm{\theta d^k}_W^2
            \le \frac{1}{2}\norm{u^k - \optu}_M^2 + \frac{\rho_k}{2}.
        \end{split}
    \end{equation*}
    Applying \cref{lemma:inner:product:inequality:line:search} for $u=\uintermed{k}$ and $\tilde{u}=\theta d^k$, establishes the claim.
\end{proof}

\subsection{Efficient linearised line search}
\label{sec:line-search:linearised}

\def\Phibk{\Psi^k}
\def\tildePhibk{\tilde\Psi^k}

In many applications, including the examples of \cref{sec:numerical}, the evaluation of the smooth function $E$ in the line search criterion \eqref{eq:new:line:search} can be computationally highly expensive.
To avoid the full evaluation of $E$, we will now perform line search on the $E$-linearisation of (a quadratically penalized version of) $\Phi^k$ of \eqref{eq:phik}, i.e.,
\[
    \Phibk(u)
    \defeq
    \bar{G}(u)
    +
    \bar E(\uintermed{k}) + \iprod{\grad \bar E(\uintermed{k})}{u-\uintermed{k}}
    +
    \iprod{\Xi \uintermed{k}}{u}.
\]
To prove that this works, we exploit the descent inequality
\begin{equation}
    \label{eq:descent-inequality}
    E(x+h)
    \le
    E(x)
    +
    \iprod{\grad E(x)}{h}
    +
    \frac{L}{2}\norm{h}^2,
\end{equation}
which holds when $E$ has $L$-Lipschitz gradient, even without convexity \cite[Chapter 7]{clason2020introduction}.

\begin{lemma}
    \label{lemma:line:search:surrogate}
    Assume that the conditions of \cref{lemma:line:search} hold.
    Let $\varepsilon_k,\rho_k>0$, and $\kappa\in (0,1)$.
    With $d^k$ defined by \cref{corollary:descent:direction}, there then exists  $\theta_0> 0$ such that for all $\theta \in [0, \theta_0]$, we have
    \begin{equation}
        \label{eq:new:line:search:surrogate}
        \Phibk(\uintermed{k}+\theta d^k)
        +
        \frac{1}{2}\norm{\theta d^k}_{M+\Lambda}^2
        +
        \frac{1}{2\varepsilon_k}
        \norm{\theta d^k}_W^2
        \le
        \Phibk(\uintermed{k})
        +
        \kappa\theta
        [\Phibk]'(\uintermed{k};d^k) + \frac{\rho_k}{2}.
    \end{equation}
\end{lemma}

\begin{proof}
    \Cref{corollary:descent:direction} shows that $d^k = I_H^h(\tilde v^{k,m}-\dualv{0})$ is a descent direction for $\Phi^k$.
    Let
    \[
        \tildePhibk(u)
        \defeq
        \Phibk(u)
        +
        \frac{1}{2}\norm{u-\uintermed{k}}_{M+\Lambda}^2
        +
        \frac{1}{2\varepsilon_k}
        \norm{u-\uintermed{k}}_W^2.
    \]
    Then
    $
        [\tildePhibk]'(\uintermed{k};d^k)
        =
        [\Phibk]'(\uintermed{k};d^k)
        =
        [\Phi^k]'(\uintermed{k};d^k).
    $
    Therefore, \eqref{eq:new:line:search:surrogate} rewrites as
    \[
        \tildePhibk(\uintermed{k}+\theta d^k)
        \le
        \tildePhibk(\uintermed{k})
        +
        \kappa\theta
        [\tildePhibk]'(\uintermed{k};d^k) + \frac{\rho_k}{2}.
    \]
    Now we simply use \cref{lemma:line:search} on $\tildePhibk$.
\end{proof}

This result allows us to formulate an efficient line search procedure.

\begin{theorem}
    \label{theorem:efficient:line:search}
    Assume that $\grad E$ is $L$-Lipschitz. Then \eqref{eq:new:line:search:surrogate} implies the sufficient descent condition \eqref{eq:new:line:search} of \cref{lemma:improve:line:search}.
\end{theorem}

\begin{proof}
    Since $\Phibk(\uintermed{k})=\Phi^k(\uintermed{k})$ as well as $[\Phibk]'(\uintermed{k};d^k)= [\Phi^k]'(\uintermed{k};d^k)$, it suffices to show that
    $
        \Phi^k(\uintermed{k} + \theta d^k)\le \Phibk(\uintermed{k} + \theta d^k) + \frac{1}{2} \norm{\theta d^k}_{\Lambda}^2.
    $
    But this is immediate from to the descent inequality \eqref{eq:descent-inequality},
    that is
    $
        \bar{E}(\uintermed{k} +\theta d^k)
        \le
        \bar{E}(\uintermed{k})
        +
        \iprod{\grad \bar{E}(\uintermed{k})}{\theta d^k}
        +
        \frac{1}{2}
        \norm{\theta d^k}_\Lambda^2.
    $
\end{proof}

\begin{corollary}
    \label{cor:lagrangian:gap:descent:convergence-light}
    We can replace \eqref{eq:new:line:search} by \eqref{eq:new:line:search:surrogate} in  \cref{theorem:lagrangian:gap:descent:convergence}.
\end{corollary}

\section{Convergence analysis}
\label{sec:convergence}

We finally prove the convergence of \cref{alg:PDMCC}.
We start in \cref{sec:convergence:inequality} by presenting a bound on the ergodic gap of the fine problem.
This bound still depends on a uniform bound on $\{\norm{\uintermed{k}}\}_{k \in \N}$, which we derive in \cref{sec:convergence:uniform-bound}.
With that in hand, we can can then, in \cref{sec:convergence:convergence}, establish the ergodic convergence of the Lagrangian gap, and, via Féjer quasi-monotonicity and Opial's lemma, the weak convergence of the iterates.
We also illustrate in \cref{rem:nonconvex}, how our results can be extended to nonconvex $E$.

\subsection{A preliminary estimate}
\label{sec:convergence:inequality}

The next lemma almost shows ergodic gap convergence, however, the right hand side will still need to be bounded.
There, we set
\[
    \triggerix \defeq \{k \in \N \mid \text{the trigger condition of \cref{alg:PDMCC} holds on iteration}\,k\}.
\]

\begin{lemma}
    \label{lemma:main:almost}
    Let $\{\varepsilon _k\}_{k\in \triggerix}$ and $\{\rho _k\}_{k\in \triggerix}$ be two sequences
    such that $\varepsilon_k,\rho_k >0$ for all $k\in \triggerix$.
    Suppose that \cref{assumption:basic:fine:structure,assumption:coarse:basic,assumption:coherence:condition} hold.
    Then, for any initial iterate $u^0=(x^0,y^0)\in U$, the sequence $\{u^k\}_{k\in \N}$ generated by \cref{alg:PDMCC} satisfies
    \begin{equation}
        \label{eq:main:result}
        \sum _{k=0}^{N-1} \gap _L(u^{k+1};\optu) + \frac{1}{2}\norm{u^N-\optu}_M^2
        +\frac{1}{2} \sum_{k=0}^{N-1} \norm{\uintermed{k}-u^k}_{M-\Lambda}^2
        \le
        \frac{1}{2}\norm{u^0-\optu}_M^2 + \frac{1}{2} \sum_{k=0}^{N-1}  \gamma_{k+1}
    \end{equation}
    for all $\optu \in U$ and
    \begin{equation}
        \label{eq:sequence:main:result}
        \gamma_{k+1} \defeq
        \begin{cases}
            0,                                                     & k\notin \triggerix,
            \\
            \varepsilon _k \norm{\uintermed{k}-\optu}_M^2 +\rho_k, & k\in \triggerix.
        \end{cases}
    \end{equation}
\end{lemma}

\begin{proof}
    \Cref{alg:PDMCC} alternates between fine PDPS iterations and coarse corrections.
    By \cref{thm:basic:PDPS-inequality,theorem:lagrangian:gap:descent:convergence,cor:lagrangian:gap:descent:convergence-light}
    \begin{equation}
        \label{eq:lagrangian:gap:update:PDMCC}
        \gap_L(u^{k+1};\optu) + \frac{1}{2}\norm{u^{k+1}-\optu}_M^2 + \frac{\mu_{k+1}}{2}\le  \frac{1}{2}\norm{u^k-\optu}_M^2 +\frac{\gamma_{k+1}}{2}
        \quad\text{for all}\quad k \in \N
    \end{equation}
    and
    \[
        \mu_{k+1} \defeq
        \begin{cases}
            \norm{u^{k+1}-u^k}_{M-\Lambda}^2,       & k\notin \triggerix,
            \\
            \norm{\uintermed{k}-u^k}_{M-\Lambda}^2, & k\in \triggerix.
        \end{cases}
    \]
    However, by \cref{line:PDPS:update}, we have $\mu_{k+1} = \norm{\uintermed{k}-u^k}_{M-\Lambda}^2$ for any $k\notin\triggerix$, and hence for all $k\ge0$. Summing \eqref{eq:lagrangian:gap:update:PDMCC} over $k=0,\ldots,N-1$ establishes the claim.
\end{proof}

\subsection{Uniform boundedness}
\label{sec:convergence:uniform-bound}

For \eqref{eq:main:result} to establish convergence of the Lagrangian gaps, we now need to bound
$
    \sum _{k\in \triggerix} \varepsilon_k\norm{\uintermed{k}-\optu}_M^2 + \rho_k.
$
This follows when $\{\uintermed{k}\}_{k\in\triggerix}$ is uniformly bounded, which is what we now prove.
We start with technical results on real sequences.

\begin{lemma}
    \label{lemma:productoria:y:sumatoria}
    Let $\{\varepsilon_k\}_{k\in\N} \subset (0, \infty)$ satisfy $\sum_{k\in \N} \varepsilon_k <\infty$.
    Then $\prod _{k\in \N} (1+\varepsilon_k)<\infty$.
\end{lemma}

\begin{proof}
    Since the exponential function is convex, we have $1+x \le \exp(x)$ for all $x\in \mathbb{R}$.
    In addition, given that $1+\varepsilon_k > 0$ for all $k \ge 0$, we obtain $\prod_{k=0}^n (1+\varepsilon_k) \le \exp\big(\sum_{k=0}^n \varepsilon_k\big)\le \exp\big(\sum_{k \in \mathbb{N}} \varepsilon_k\big)$.
    Therefore, the sequence of partial products is uniformly bounded, and consequently, the desired result holds.
\end{proof}

\begin{lemma}
    \label{lemma:recursividad:sucesion}
    Let $\{\varepsilon_i\}_{i\in \N}, \{\rho_i\}_{i\in \N} \subset (0, \infty)$ satisfy $\sum_{i\in \N}\varepsilon_i <\infty$ and $\sum_{i\in \N}\rho_i <\infty$. Let $\{\omega_k\}_{k\in \N} \subset (0, \infty)$ be such that
    \begin{equation}
        \label{eq:hipotesis:recursividad}
        \omega_{k+1} \le  \varpi_k \defeq \omega_0 + \sum _{i=0}^k [\varepsilon_i \omega_i + \rho_i]
    \end{equation}
    for all $k \ge 0$ and a $\omega_0\ge 0$. Then,
    $
        \omega_k \le (\omega_0+ \sum_{i\in \N}\rho_i) [\prod_{i\in\N}(1+\varepsilon_i)] <\infty
    $
    for all $k \ge 0$.
\end{lemma}

\begin{proof}
    Since $\sum_{i\in\N}\varepsilon_i<\infty$ and $\sum_{i\in\N}\rho_i<\infty$, \cref{lemma:productoria:y:sumatoria} guarantees that $\prod_{i=0}^k (1+\varepsilon_i)\le \prod_{i\in \N}(1+\varepsilon_i)<\infty$ and also $\sum_{i=0}^k \rho_i  \le \sum_{i\in\N}\rho_i<\infty$.
    Thus, our claim follows if we prove that $\varpi_k \le \beta_k$, where we define and estimate
    \[
        \beta_k\defeq\omega_0\prod_{i=0}^k(1+\varepsilon_i) + \sum _{i=0}^k \rho_i\prod_{p=i+1}^k (1+\varepsilon_p)
        \le (\omega_0 + \sum_{i\in \N}\rho_i)\left[\prod_{i\in\N} (1+\varepsilon_i)\right].
    \]
    The proof is by induction. For $k = 0$, the definition of $\varpi_0$ yields
    \[
        \varpi_0 = \omega_0 + \varepsilon_0 \omega_0 + \rho_0 \le \omega_0 + \varepsilon_0 \omega_0 + \rho_0 = \omega_0(1+\varepsilon_0) + \rho_0 = \beta_0.
    \]
    Now suppose it holds for $k-1$.
    Firstly, based on the definition of $\beta_k$, we have
    \[
        \beta_k = \left[\omega_0 \prod_{i=0}^{k-1}(1+\varepsilon_i)+\sum_{i=0}^{k-1}\rho_i \prod_{p=i+1}^{k-1}(1+\varepsilon_p)\right](1+\varepsilon_k) + \rho_k = \beta_{k-1}(1+\varepsilon_k) +\rho_k,
    \]
    for all $k \ge 1$. Now, by \cref{eq:hipotesis:recursividad}, we have $\omega_k \le \varpi_{k-1} \le \beta_{k-1}$.
    Then,
    \begin{equation*}
        \varpi_k = \omega_0 + \sum _{i=0}^k [\varepsilon_i \omega_i +\rho_i] = \varpi_{k-1} + \varepsilon_k\omega_k + \rho_k
        \le \beta_{k-1}(1+\varepsilon_k) + \rho_k = \beta_k.
        \qedhere
    \end{equation*}
\end{proof}

The next result establishes the required uniform bound, subject to the following control on the line search model parameters.

\begin{assumption}
    \label{assumption:unoform:boundedness}
    $\{\varepsilon_k\}_{k\in \triggerix} \subset (0, \infty)$ and $\{\rho_k\}_{k\in \triggerix} \subset (0, \infty)$ satisfy
    \[
        \sum_{k\in \triggerix}\varepsilon_k \defeq \varepsilon<\infty,
        \quad
        \sum_{k\in \triggerix}\rho_k \defeq \rho<\infty,
        \quad
        \text{and}
        \quad
        \prod_{k\in \triggerix}(1+\varepsilon) \defeq \kappa_{\varepsilon}<\infty.
    \]
\end{assumption}

\begin{corollary}
    \label{corollary:pre-fine-grid:bounded}
    Let \cref{assumption:basic:fine:structure,assumption:coarse:basic,assumption:coherence:condition,assumption:unoform:boundedness} hold.
    Then for any initial $u^0 \in \Omega$ and any primal-dual solution $\optu \in \opt U \defeq [\partial \bar{G} + \grad\bar{E} + \Xi]^{-1}(0)$ of the fine problem \cref{eq:general:optimization:problem}, we have
    \begin{equation}
        \label{eq:uniform:boundedness:u:intermed}
        \sup _{k\in \triggerix} \norm{\uintermed{k}-\optu}_M^2 \le \kappa_{\varepsilon}(\norm{u^0-\optu}_M^2 + \rho).
    \end{equation}
    Furthermore,
    \begin{equation}
        \label{eq:bound:of:sum}
        \sum_{k=0}^{N-1}\gamma_{k+1} \le \sum _{k\in\triggerix} \varepsilon_k \norm{\uintermed{k}-\optu}_M^2 + \rho_k \le \kappa_{\varepsilon}(\norm{u^0-\optu}_M^2 + \rho) \varepsilon + \rho < \infty.
    \end{equation}
\end{corollary}

\begin{proof}
    We order the trigger iteration indices $k_n \in \triggerix$ as $k_0<k_1<k_2<\ldots$.
    Since \cref{line:standar:PDPS} of \cref{alg:PDMCC} performs a standard PDPS pre-step performed each coarse correction, \cref{thm:basic:PDPS-inequality} yields
    \begin{equation}
        \label{eq:recursividad:Ji:coarse}
        \norm{\uintermed{k_n}-\optu}_M^2\le\norm{u^{k_n}-\optu}_M^2 \quad \text{for all}\quad n \ge 0.
    \end{equation}
    Moreover, after updating the fine variable with the coarse correction \Cref{line:ls:update}, PDPS steps are performed until reaching the next coarse correction.
    Thus, \eqref{eq:lagrangian:gap:update:PDMCC} from \cref{lemma:main:almost} holds with $\gamma_{k+1} = 0 $ for all the non-trigger iterations $k = k_{n-1} +1,\ldots,k_n-1$ for every $n\ge 0$, where we set $k_{-1} \defeq -1$.
    By \cref{assumption:basic:fine:structure,lemma:positive:define:M}, we have $\gap_L(u^{k+1};\optu)\ge 0$ and $M\ge \Lambda$, which implies that $\mu_{k+1}\ge 0$.
    Consequently,
    \begin{equation}
        \label{eq:recursividad:Ji}
        \norm{u^{k_0}-\optu}_M^2 \le \norm{u^0-\optu}_M^2\quad
        \text{and}\quad
        \norm{u^{k_n}-\optu}_M^2 \le \norm{u^{k_{n-1}+1}-\optu}_M^2,
        \quad \text{for all}\quad n \ge 1.
    \end{equation}
    On each trigger iteration $k_{n-1} \in \triggerix$, ($n\ge 1$), it follows from the definition of $\gamma_{k_{n-1}+1}$ in \eqref{eq:lagrangian:gap:update:PDMCC}, from \cref{lemma:main:almost}, that
    \begin{equation}
        \label{eq:recursividad:coarse:update}
        \norm{u^{k_{n-1}+1}-\optu}_M^2 \le \norm{u^{k_{n-1}}-\optu}_M^2 + \varepsilon_{k_{n-1}}\norm{\uintermed{k_{n-1}}-\optu}_M^2 + \rho_{k_{n-1}},
        \quad \text{for all}\quad n \ge 1.
    \end{equation}
    Combining the second inequality of \cref{eq:recursividad:Ji} with \cref{eq:recursividad:coarse:update} yields
    \[
        \norm{u^{k_n}-\optu}_M^2 \le \norm{u^{k_{n-1}}-\optu}_M^2 + \varepsilon_{k_{n-1}}\norm{\uintermed{k_{n-1}}-\optu}_M^2 + \rho_{k_{n-1}}
        \quad \text{for all}\quad n \ge 1.
    \]
    Summing this over $n=1,\ldots,N$ and using \cref{eq:recursividad:Ji:coarse} and the first inequality of \cref{eq:recursividad:Ji}, we obtain
    \[
        \norm{\uintermed{k_N}-\optu}_M^2
        \le
        \norm{u^{k_N}-\optu}_M^2
        \le
        \norm{u^0-\optu}_M^2 + \sum _{i=0}^{N-1} [\varepsilon_{k_i}\norm{\uintermed{k_i}-\optu}_M^2 + \rho_{k_i}].
    \]
    This reads as
    $
        \omega_N \le \omega_0 + \sum_{i=0}^{N-1}[\varepsilon_{k_i}\omega_i + \rho_{k_i}]
    $
    for  $\omega_i \defeq \norm{\uintermed{k_i}-\optu}_M^2$ and $\omega_0 \defeq \norm{u^0-\optu}_M^2$.
    \Cref{assumption:unoform:boundedness,lemma:recursividad:sucesion} show the uniform boundedness of the sequence $\{\uintermed{k}-\optu\}_{k\in \triggerix}$, i.e., \cref{eq:uniform:boundedness:u:intermed}.
    Since $\{\uintermed{k}\}_{k\in \triggerix}$ is uniformly bounded, it follows that
    \[
        \sum_{k=0}^{N-1} \gamma_{k+1}
        =
        \sum_{k\in \triggerix_N} \varepsilon _k \norm{\uintermed{k}-\optu}_M^2 +\rho_k
        \le \sup_{k\in \triggerix}\norm{\uintermed{k}-\optu}_M^2 \sum_{k\in \triggerix_N}[\varepsilon_k + \rho_k],
    \]
    where $\triggerix_N \defeq \{0,\ldots,N-1\}\cap \triggerix$.
    Consequently, \cref{eq:uniform:boundedness:u:intermed} implies \cref{eq:bound:of:sum}.
\end{proof}

\subsection{Convergence}
\label{sec:convergence:convergence}

We can now prove the Fejér quasi-monotonicity of the sequence generated by \cref{alg:PDMCC}, and establish the ergodic convergence of the Lagrangian gap, as well as the weak convergence of the iterates.

\begin{corollary}
    \label{cor:ergodic:convergence:PDMCC}
    Assume that \cref{assumption:basic:fine:structure,assumption:coarse:basic,assumption:coherence:condition,assumption:unoform:boundedness} hold.
    Then for any initial $u^0 \in U$ the iterates generated by \cref{alg:PDMCC} satisfy for any $\optu\in \defeq [\partial \bar{G}+ \grad \bar{E} + \Xi]^{-1}(0)$ the ergodic gap estimate
    \begin{equation}
        \label{eq:ergodic:estimate:PDMCC}
        \gap _L(\tilde u^N;\optu)\le \frac{\varepsilon\kappa_{\varepsilon}+1}{2N}\left[\norm{u^0-\optu}_M^2 + \rho\right]
        \quad\text{where}\quad
        \tilde u^N \defeq \frac{1}{N}\sum_{k=0}^{N-1} u^{k+1}.
    \end{equation}
    In particular $\gap _L(\tilde u^N;\hat U) \to 0$ at the rate $O(1/N)$.
\end{corollary}

\begin{proof}
    Combining \cref{corollary:pre-fine-grid:bounded,lemma:main:almost} and using $M\ge \Lambda$ from \cref{lemma:positive:define:M}, we get
    \[
        \sum _{k=0}^{N-1} \gap_L(u^{k+1};\optu)
        \le \frac{1}{2}
        \big[\norm{u^0-\optu}_M^2 + \kappa_{\varepsilon}(\norm{u^0-\optu}_M^2+\rho)\varepsilon + \rho\big]
        =
        \frac{\varepsilon \kappa_{\varepsilon}+1}{2}\big[\norm{u^0-\optu}_M^2 + \rho\big]<\infty.
    \]
    An application of Jensen's inequality now establishes \cref{eq:ergodic:estimate:PDMCC}.
\end{proof}

\begin{theorem}
    \label{theorem:convergence:by:Opial}
    Assume that \cref{assumption:basic:fine:structure,assumption:coarse:basic,assumption:coherence:condition,assumption:unoform:boundedness} hold.
    Suppose $[\partial \bar{G}+ \grad \bar{E} + \Xi]^{-1}(0)\neq \emptyset$, i.e., the fine problem \cref{eq:general:optimization:problem} has a minimizer.
    Then for any initial $u^0 \in U$ the sequence generated by \cref{alg:PDMCC} is quasi-Fejér monotone, and converges weakly to a root $\optu \in [\partial \bar{G}+ \grad \bar{E} + \Xi]^{-1}(0)$.
\end{theorem}

\begin{proof}
    We note $\opt U \defeq [\partial \bar{G}+ \grad \bar{E} + \Xi]^{-1}(0)$.
    Let $\optu\in \hat U$ be an optimal solution to the problem \cref{eq:general:optimization:problem}.
    Let $\gamma_{k+1}$ be given by \eqref{eq:sequence:main:result}.
    Since $\gap _L(u;\optu)\ge 0 $ for all $u\in U$ and $M\ge \Lambda$, the inequality \cref{eq:lagrangian:gap:update:PDMCC} reduces to
    \[
        \norm{u^{k+1}-\optu}_M^2 \le  \norm{u^k-\optu}_M^2 +\gamma_{k+1}
        \quad \text{for all}\quad k \ge 0.
    \]
    We already know from \cref{corollary:pre-fine-grid:bounded} that $\sup_{k\in\triggerix} \norm{\uintermed{k}-\optu}_M^2<\infty$,
    hence
    \[
        \gamma _{k+1} \le \tilde{\gamma}_{k+1} =
        \begin{cases*}
        0                                                                         & $k\in \triggerix$,
        \\
        \kappa_{\varepsilon}(\norm{u^0 -\optu}_M^2 + \rho)\varepsilon_k +\rho_k & $k\notin \triggerix$.
        \end{cases*}
    \]
    Since \cref{assumption:unoform:boundedness} holds, it follows that $\sum_{k\in \N} \tilde \gamma_{k+1} <\infty$. Thus
    the sequence $\{u^k\}_{k\in \N}$ exhibits quasi-Fejér monotonicity respect to $\opt U$.

    By \cref{eq:lagrangian:gap:update:PDMCC} from \cref{lemma:main:almost}, we have,
    for all  $k \in \N$ and $\optu \in \opt U$,
    \[
        \gap_L(u^{k+1};\optu) + \frac{1}{2}\norm{u^{k+1}-\optu}_M^2
        + \frac{1}{2}\norm{\uintermed{k}-u^k}_{M-\Lambda}^2 \le \frac{1}{2}\norm{u^k-\optu}_M^2 + \frac{\gamma_{k+1}}{2}.
    \]
    Since $\gap _L(u^k;\optu)\ge0$ and summing over $k=0,\ldots,N-1$, we obtain
    \[
        \sum _{k=0}^{N-1} \norm{\uintermed{k}-u^k}_{M-\Lambda}^2 + \norm{u^N-\optu}_M^2 \le  \norm{u^0-\optu}_M^2 + \sum_{k=0}^{N-1} \gamma_{k+1}.
    \]
    Using \cref{eq:bound:of:sum} from \cref{corollary:pre-fine-grid:bounded}, we have
    \[
        \sum _{k=1}^{N-1} \norm{\uintermed{k}-u^k}_{M-\Lambda}^2 \le
        \frac{\varepsilon \kappa_{\varepsilon}+1}{2}[\norm{u^0-\optu}_M^2 + \rho]<\infty.
    \]
    This proves that
    $
        \norm{\uintermed{k}-u^k}_{M-\Lambda}\to0.
    $
    By \cref{assumption:basic:fine:structure,lemma:positive:define:M}, we conclude that
    $
        \uintermed{k}-u^k\to 0.
    $
    Let $\optu$ be a limit point of $\{\uintermed{k}\}_{k\in \N}$, i.e., there exists a subsequence $\{k_i\}_{i \in \N}$ such that $\uintermed{k_i}\to \optu$. To prove that $\optu \in \opt U$, we consider the implicit equation, related to \cref{line:standar:PDPS} of \cref{alg:PDMCC},
    \[
        0 \in \partial \bar G(\uintermed{k_i}) + \Xi \uintermed{k_i} + \grad \bar E(u^{k_i}) + M(\uintermed{k_i}-u^{k_i}).
    \]
    Since $\partial \bar G + \grad \bar E + \Xi $ is weak-to-strong outer semicontinous, $M$ is self-adjoint, bounded and positive definite, and $\uintermed{k}-u^k\to 0$, we conclude that $\optu \in \opt U$, (cf.~\cite[Chapter 9]{clason2020introduction}).
    However, since $\uintermed{k} - u^k\to 0$ it follows that $u^{k_i} = (u^{k_i}-\uintermed{k_i}) + \uintermed{k_i} \to  \optu$, i.e.,
    both sequences, $\{u^k\}_{k\in\N }$ and $\{\uintermed{k}\}_{k\in \N}$, share the same set of limit points, and consequently all limit points of $\{u^k\}_{k\in\N }$ are solutions of the problem.
    Finally, applying Opial’s lemma for quasi-Fejér sequences \cite[Lemma A.2.]{tuomov2024tracking}, the result follows.
\end{proof}

The next result shows that also the coarse corrections converge, to zero.

\begin{corollary}
    Let \cref{assumption:basic:fine:structure,assumption:coarse:basic,assumption:coherence:condition,assumption:unoform:boundedness} hold.
    Suppose $\opt U \defeq [\partial \bar{G}+ \grad \bar{E} + \Xi]^{-1}(0)\neq \emptyset$. Then the coarse-grid correction vanishes asymptotically, i.e., $d^k \to 0$.
\end{corollary}

\begin{proof}
    By \cref{corollary:descent:direction}, we have
    \begin{equation}
        \label{eq:direction:bound}
        \norm{d^k}
        \le
        \tilde c_Q
        \inf_{g\in\partial\bar G(\uintermed{k})}
        \norm{g+\grad\bar E(\uintermed{k})+\Xi\uintermed{k}}
        \quad\text{for all}\quad k\in\triggerix.
    \end{equation}
    On the other hand, by \cref{line:standar:PDPS} of \cref{alg:PDMCC}, there exists $g^k\in\partial\bar G(\uintermed{k})$ such that
    \[
        g^k
        +
        \grad\bar E(\uintermed{k})
        +
        \Xi\uintermed{k}
        =
        \grad\bar E(\uintermed{k})
        -
        \grad\bar E(u^k)
        -
        \tau^{-1}M(\uintermed{k}-u^k).
    \]
    Since $\grad E$ is $L$-Lipschitz continuous and, by \cref{assumption:basic:fine:structure}, $M$ is bounded, combining this identity with \eqref{eq:direction:bound}, we obtain
    \[
        \norm{d^k}
        \le
        \tilde c_Q
        \norm{\grad\bar E(\uintermed{k})-\grad\bar E(u^k)-\tau^{-1}M(\uintermed{k}-u^k)}
        \le
        \tilde c_Q
        (L+\tau^{-1}\norm{M})
        \norm{\uintermed{k}-u^k}.
    \]
    We finish by observing from the proof of \cref{theorem:convergence:by:Opial} that $\uintermed{k}-u^k \to 0$
\end{proof}

\begin{remark}[Nonconvex fine problems]
    \label{rem:nonconvex}
    The convexity of $E$ has been required only in \cref{thm:basic:PDPS-inequality}, and in \cref{corollary:pre-fine-grid:bounded} to have $\gap_L(u^{k+1};\optu)\ge0$.
    In the first case, this arises through a three-point descent inequality \cite[Chapters 7 and 11]{clason2020introduction}
    \[
        \iprod{\grad E(\thisx)}{\nextx-\optx} \ge E(\nextx) - E(\optx) - \frac{L}{2}\norm{\nextx-\thisx}^2,
    \]
    where $L$ is the Lipschitz factor of $\grad E$.
    The two-point descent inequality \eqref{eq:descent-inequality}, i.e., $\optx=\thisx$ here, holds even without convexity.
    The three-point version holds for nonconvex functions provided $\thisx$ is in a local neighborhood of $\optx$ with some second-order growth \cite{tuomov-proxtest,tuomov2024tracking}.
    There is no requirement for $\nextx$ to be in that neighborhood on iteration $k$.

    However, ensuring that $\gap_L(u^{k+1};\optu)\ge0$ is somewhat more involved.
    We need to repeat the arguments of the proof \emph{a priori} without involving the gap, using monotonicity and three-point co-coercivity.
    Though the summability of $\tilde\gamma_k$, we can then obtain an a priori bound on $\norm{\nextx-\optx}$, which will guarantee that we stay in a given local neighborhood, if we start close enough to a solution.
    We can then improve this bound by repeating the above gap-based arguments \emph{a posteriori}.
    See \cite{tuomov-pdex2nlpdhgm,tuomov2024tracking,tuomov2024online-eit} for details.

    With this, we can, locally, extend our gap convergence results to nonconvex $E$.
    In \cite[§7.3]{tuomov2024tracking}, it is also shown how to obtain convergence of the convex envelope from gap convergence.
    Weak convergence of iterates iterates requires, e.g., explicitly assuming weak-to-strong continuity of $\grad E$; compare  \cite{tuomov-pdex2nlpdhgm} and \cite[§5.6]{tuomov2024tracking}.
\end{remark}

\section{Construction of coarse functions}
\label{sec:construction}

We now provide several examples on the construction of the smooth coarse function $E_H$ (\cref{sec:construction:smooth}), and recall the approach of \cite{guerra2025multigrid} for the nonsmooth component functions $(G_H^*)^k$ and $F_H^k$ (\cref{sec:construction:nonsmooth}).

\subsection{Data term}
\label{sec:construction:smooth}

In inverse problems applications, the smooth function $E$ is typically a data term, which involves expensive-to-evaluate operators mapping a desired reconstruction to measurable data.
An important objective in the construction of the coarse variant $E_H$ is to reduce the operator evaluation cost.
Our first “ideal” example often fails this, as it requires evaluating the original fine function.

\begin{example}[Reparametrisation of the fine function]
    \label{ex:construction:smooth:ideal}
    Given the initial primal coarse iterate $\zeta^{k,0} = P_h^H \xintermed{k}$, the conceptually ideal choice for $E_H:X_H\to \extR$ is
    \[
        E_H(\zeta) \defeq E(\xintermed{k} + P_H^h(\zeta - \zeta^{k,0})).
    \]
    That is, we use the original $E$, adding the restriction error $\xintermed{k} - P_H^h\zeta^{k,0}$ to the prolongation of $\zeta$.
    Then, at the initial iterate, $\grad E_H(\zeta^{k,0}) = P_h^H\grad E(\xintermed{k})$.
    Thus, in \cref{alg:coarse},
    $
        a_H^k = P_h^H K^* \yintermed{k}.
    $
    For the primal coarse step on \cref{line:coarse:primal} of \cref{alg:coarse}, we use the definition \eqref{eq:coarse:ehk} of $E_H^k$ to construct
    \[
        \begin{cases*}
            r_H^k = P_h^H(\grad E(\xintermed{k}) + K^*\yintermed{k})-(\grad E_H(\zeta^{k,0}) + K_H^*(\xi^{k,0})) = P_h^HK^*\yintermed{k} - K_H^*\xi^{k,0}     & ergodic,
            \\
            r_H^{k,j} = P_h^H(\grad E(\xintermed{k}) + K^*\yintermed{k})-(\grad E_H(\zeta^{k,0}) + K_H^*(\xi^{k,j})) = P_h^HK^*\yintermed{k} - K_H^*\xi^{k,j} & non-ergodic.
        \end{cases*}
    \]
    Hence, the gradient of $E_H^k$ reads
    \begin{equation}
        \label{eq:grad:EHk:ideal}
        \grad E_H^k(\zeta) =
        \begin{cases*}
            P_h^H\grad E(\xintermed{k}-P_H^h(\zeta -\zeta^{k,0})) + P_h^H K^* \yintermed{k} - K_H^*\xi^{k,0} & ergodic,
            \\
            P_h^H\grad E(\xintermed{k}-P_H^h(\zeta -\zeta^{k,0})) + P_h^H K^* \yintermed{k} - K_H^*\xi^{k,j} & non-ergodic.
        \end{cases*}
    \end{equation}
    Moreover, $\grad E_H$ is $L_H=\norm{P_H^h}\norm{P_h^H}L$-Lipshitz, if $\grad E$ is $L$-Lipschitz, so the relevant parts of \cref{assumption:coarse:basic} hold.
\end{example}

\begin{example}[Linearisation]
    \label{ex:construction:smooth:linearistion}
    In the previous example, computing $\grad E_H(\zeta)$ for $\zeta \neq \zeta^{k,0}$ can be expensive.
    For this reason, we introduce its linearisation
    \[
        E_H(\zeta) \defeq E(\xintermed{k}) + \iprod{\grad E(\xintermed{k})}{P_H^h(\zeta -\zeta^{k,0})}
    \]
    Again, $\grad E_H(\zeta^{k,0}) = P_h^H\grad E(\xintermed{k})$ holds, and $r_H^k $ and $r_H^{k,j}$ coincide with \cref{ex:construction:smooth:ideal}.
    Now, $\grad E_H^k$ is Lipshitz with factor $L_H=0$, and
    \[
        \grad E_H^k(\zeta) =
        \begin{cases*}
            P_h^H(\grad E(\xintermed{k}) + K^*\yintermed{k}) - K_H^*\xi^{k,0} & ergodic,
            \\
            P_h^H(\grad E(\xintermed{k}) + K^*\yintermed{k}) - K_H^*\xi^{k,j} & non-ergodic.
        \end{cases*}
    \]
\end{example}

\begin{example}[Quadratic data terms]
    \label{ex:construction:smooth:quadratic}
    For $A \in \linear(X; V)$ and data $e \in V$, consider
    \[
        E(x) \defeq \frac{1}{2}\norm{Ax-e}_V^2
    \]
    in an Euclidean space $V$.
    It seems then reasonable to take $E_H$ of the same form,
    \[
        E_H(\zeta) \defeq \frac{1}{2}\norm{A_H \zeta - e_H}_{V_H}^2
    \]
    for some linear operator $A_H \in \linear(X_H; V_H)$ and data $e_H \in V_H$ in an Euclidean space $V_H$.
    Recall from \cref{ex:construction:smooth:ideal} that an ideal choice of $E_H$ would generally be
    \[
        E_H(\zeta) \defeq E(\xintermed{k} + I_H^h (\zeta -\zeta^{k,0})) = \frac{1}{2}\norm{A(\xintermed{k} + P_H^h (\zeta -\zeta^{k,0})) - e}_V^2
    \]
    i.e. $e_H= e +P_H^h\zeta^{k,0} -A\xintermed{k}$ and  $A_H = A P_H^h$.
    Computing $\grad E_H = A_H^*(A_H \zeta  - e_H)$ only requires one application of $A_H^*A_H$ per iteration and a single pre-computation of the coarse data $A_H^*e_H$.
    However, we do not wish to compute the possibly expensive $A$, so require an efficient presentation for $AP_H^h$, or an $A_H$ that approximates $A P_H^h$.
\end{example}

\subsection{Nonsmooth functions}
\label{sec:construction:nonsmooth}

The construction of the coarse functions $F_H^k$ and $(G_H^k)^*$ that satisfy \cref{assumption:coherence:condition} can be carried out using polar cone indicators. For a set $A\subset X$, we recall that the polar cone $A^\circ\defeq \{z \in X \mid \iprod{z}{x}\le 0 \text{ for all } x\in X\}$ and the bipolar cone $A^{\circ\circ} \defeq (A^\circ)^\circ$.
We have $A^{\circ\circ}\supset A$ with equality if $A$ is non-empty, convex, and closed \cite[Theorem 1.8]{clason2020introduction}.

\begin{lemma}
    \label{lemma:construction:nonsmooth-polar}
    Take $\zeta^{k,0}=P_H^h\xintermed{k}$, where $\xintermed{k} \in \Dom F$, and $F$ is convex, proper, and lower semicontinuous.
    Then $P_h^H \partial F(\xintermed{k}) \subset \partial F_H^k(\zeta^{k,0})$ for $F_H^k = \delta_{\Omega^k}$ with $\Omega^k \defeq \zeta^{k,0} + (P_h^H\partial F(\xintermed{k}))^\circ$.
\end{lemma}

\begin{proof}
    Indeed, $\Omega^k$ is clearly nonempty and convex. Furthermore, since the subdifferential of the indicator function is the normal cone, $\partial F_H^k(\zeta^{k,0}) = N_{\Omega^k}(\zeta^{k,0}) = \bipolar{(P_h^H \partial F(\xintermed{k}))} \supset \polar{(P_h^H \partial F(\xintermed{k}))}$.
\end{proof}

More details on the construction can be found in \cite{guerra2025multigrid} for $G^*=\delta_{B(0, \alpha)^n}$, where $B(0, \alpha)$ is the Euclidean ball.
This arises when $G(K\freevar)$ models total variation.

\section{Numerical experiments}
\label{sec:numerical}

We now compare the two PDMCC variants against the PDPS on three inverse imaging problems: magnetic resonance imaging (MRI), positron emission tomography (PET), and electrical impedance tomography (EIT), all with total variation (TV) regularization.
All involve expensive-to-evaluate operators. The EIT problem is nonconvex.
These problems share the structure
\begin{equation}
    \label{eq:modelo:imagen:general}
    \min_{x\in X}~ J(x) \defeq F(x) + E(x) + \alpha \norm{\grad _h x}_{2,1}
\end{equation}
for a regularization parameter $\alpha>0$.
Before treating the specifics of each problem in \cref{sec:numerical:mri,sec:numerical:pet,sec:numerical:eit}, first, in \cref{sec:numerical:trigger}, we formulate the trigger condition that we use to pass to the coarse grid in \cref{alg:PDMCC}.
Then, in \cref{sec:numerical:evaluation}, we discuss general numerical setup, shared by our example problem.
We finish the paper with our conclusions from the experiments in \cref{sec:numerical:conclusions}.

Our software implementation is available on Zenodo \cite{guerra2026pdmultigrid-codes}.
The EIT component is based on \cite{valkonen-tracking-codes,jauhiainen2025online-eit-codes}.

\subsection{The trigger condition}
\label{sec:numerical:trigger}

There are different ways to formulate the trigger condition of \cref{alg:PDMCC}.
Some are discussed in \cite{parpas2017multilevel}.
Theoretically, there are no restrictions on the trigger condition, due to the fine-grid pre-iterate $\uintermed{k}$: If there is no descent on the coarse grid, $\theta$, such as when, $v^{k,0}$ already solves the coarse problem, \cref{line:ls:update} of the algorithm reduces to a standard fine-grid PDPS step.

We use a \term{switching strategy}, maintaining a \emph{counter} $k_{\mathrm{switch}} \in \N$ and a \emph{switch} that can be set to \emph{coarse-priority} or \emph{fine-priority}, starting with coarse-priority.
The trigger condition is satisfied if, for a parameter $n \in \N$ (100 in our experiments):
\begin{enumerate}[label=(\alph*),nosep]
    \item The switch is on fine-priority, and $k \in k_{\mathrm{switch}} + n \N$, or
    \item The switch is on coarse-priority, and $k \not\in k_{\mathrm{switch}} + n \N $.
\end{enumerate}
After a coarse correction, we flip the switch and update $k_{\mathrm{switch}} \defeq k$ if :
\begin{enumerate}[label=(\roman*),nosep]
    \item If $\norm{d^k}_U\ge \eta$ for a parameter $\eta>0$, and the switch is on fine-priority, or
    \item if $\norm{d^k}_U <\eta$, and the switch is on coarse-priority.
\end{enumerate}

\subsection{General setup}
\label{sec:numerical:evaluation}

Our primal variable $x$ or $\zeta$ generally represent a two-dimensional image on a domain $\Omega$, and the dual variable $y$ or $\xi$ a corresponding vector field.
For MRI and PET, they are the nodal values of a finite differences scheme, and lie in $X \defeq \R^n$ (primal, fine), $X_H \defeq \R^N$ (primal, coarse), $Y \defeq \R^{2\times n}$ (dual, fine) and $Y_H \defeq \R^{2\times N}$ (dual, coarse), respectively, where $N<n$.
For EIT, the primal variables represent the nodal values of piecewise linear continuous finite elements ($P_1$), while the dual variables represent the elementwise values of piecewise constant functions on the same mesh.

To compare algorithm performance, we use the relative performance measure
\begin{equation}
    \label{eq:relative:performance}
    \mathrm{RP}_k\defeq
    \begin{cases*}
        \gap _L(u^k;\optu)/\gap _L(u^0;\optu), & for the convex MRI and PET problems,
        \\
        J(x^k)/ J(x^0),                        & for the nonconvex EIT problem.
    \end{cases*}
\end{equation}
Here the Lagrangian gap is defined in \cref{eq:Lagrangian:gap}, and $\optu=(\optx,\opty)$ is estimated by taking $500000$ iterations of the PDMCC using the parameter settings specified for each experiment in \cref{sec:numerical:mri,sec:numerical:pet}.
For EIT, we use the relative primal objective value since the Lagrangian gap is not an appropriate performance measure in the nonconvex setting (see \cref{rem:nonconvex}).
We use $\eta = 10^{-5}$ as the tolerance of the trigger condition of \cref{sec:numerical:trigger}.
Finally, with $a_v = 10^{-2}$ and $a_r = 10^{-4}$, we choose slowly decaying summable sequences satisfying \cref{assumption:unoform:boundedness}:
\[
    \varepsilon _k \defeq \frac{C \min\{\tau^{-2},\sigma^{-2}\}}{(1 + a_v k)^{1.1}}, \quad \text{and}\quad
    \rho_k \defeq \frac{\norm{d^k}_M^2 + \norm{d^k}_W^2}{(1 + a_r k)^{1.08}}.
\]

In our reports, the \emph{iteration comparison number} scales the number of coarse-grid iterations proportionally to the ratio of the numbers of nodes in the fine and coarse grids.
This provides a rough measure of the computational effort, allowing the comparison of the basic PDPS against the PDMCC. Additionally, we report the CPU time.

\subsection{MRI}
\label{sec:numerical:mri}

In MRI, the observables are Fourier transforms $\mathscr{F}$ of a two-dimensional image $x$.
Assuming complex Gaussian noise, and sampling this transform $t$ times with different subsampling masks presented by the operators $S_p \in \linear(\C^{n}; C^{m_p})$, ($p=1,\ldots,t$), we express the reconstruction problem for the data $b_p \in \C^{m_p}$ as \cref{eq:modelo:imagen:general} with
\[
    F\equiv 0,\quad \text{and}\quad E(x)=\frac12 \sum_{p=1}^t \norm{S_p\mathscr{F} x-b_p}_{\C^{m_p}}^2,
\]
Let $\Sym S$ denote the symmetrisation of $S\defeq \sum_{p=1}^t S_p^*S_p$ over positive and negative frequencies on both axes.
To avoid complex numbers, we can then rewrite \cite{guerra2025multigrid}
\[
    E(x)=\frac12 \iprod{Tx}{x}_{\R^n}-\iprod{x}{e}_{\R^n}
    \ \text{for}\
    T=\mathscr{F}^*\Sym S\,\mathscr{F}
    \ \text{and}\
    e=\sum_{p=1}^t \Re \mathscr{F}^* S_p^* b_p.
\]

We use standard multigrid transfer operators, consistent with the forward-difference structure of $\grad_h$ on a rectangular grid. The primal restriction operator is $P_h^H \defeq \mathscr{R}\otimes \mathscr{R}$, with stencil $\mathscr{R}=
    \begin{pmatrix}
        \frac12 & 1 & \frac12
    \end{pmatrix}
$ (see, e.g., \cite{briggs2000multigrid}), and the dual restriction operator is given by $D_h^H y \defeq [P_h^H y_1, P_h^H y_2]^\top$.

$E$ has the structure of \cref{ex:construction:smooth:quadratic} with $A = (\Sym S)^{1/2}\mathscr{F}$.
We take $E_H$ following that example.
The ideal coarse operator $A_H=(\Sym S_H)^{1/2}\mathscr{F}_H$ satisfies $AP_h^H=A_H$, but is computationally expensive. We take $A_H=\mathscr{F}_H$. Then $A_H^*A_H=\Id$, making the valuation of $\grad E_H$ on each coarse iteration very cheap. It has Lipschitz factor $L_H=1$.

As our ground truth image $\hat x$, we use the phantom of \cite{belzunce2018high} with resolutions $583\times493$ and $2048\times1732$.
The subsampling masks $S_p$ are formed by uniform sampling of horizontal lines of the Fourier transform.
We add complex Gaussian noise to $S_p\mathscr{F}\hat x$ to form the data $b_p$.
The noise levels, number of lines per mask, the value of the Lipschitz constant for the fine grid, and the PDMCC parameters are as follows:
\begin{center}
    \begin{tabular}{c c c c c c c c c}
        \cmidrule[\heavyrulewidth]{5-9}
                         &                          &     &       &            & \multicolumn{2}{c}{Ergodic} & \multicolumn{2}{c}{Non-ergodic}
        \\
        \cmidrule[\lightrulewidth]{5-9}
        Resolution       & $L=\norm{\Sym S}_\infty$ & $t$ & $\nu$ & Lines/mask & $m$                         & $(\tau_0,\sigma_0)$             & $m$ & $(\tau_0,\sigma_0)$
        \\
        \midrule
        $583\times493$   & 4                        & 8   & 50    & 100        & 4                           & $(0.05,\,0.85)$                 & 3   & $(0.05,\,0.75)$
        \\
        $2048\times1732$ & 5                        & 15  & 100   & 200        & 7                           & $(0.10,\,0.85)$                 & 4   & $(0.10,\,0.75)$
        \\
        \bottomrule
    \end{tabular}
\end{center}
We have $\norm{\nabla_h}_{\mathbb{L}(X)}, \norm{\nabla_H}_{\mathbb{L}(X)}\le \sqrt{8}$ \cite{chambolle2004algorithm}.
We take as the step length parameters $\tau=0.9/L$, $\sigma=0.09/(8\tau)$ on the fine grid, and $\tau_H=\tau_0/L$, $\sigma_H=\sigma_0(1-\tau_0)/(8\tau_H)$ on the coarse grid, with $\tau_0,\sigma_0\in(0,1)$ as in the table. We set $\alpha=0.8$.

The reconstructions are shown in \cref{fig:mri:graphs:solutions}, while performance is reported in \cref{fig:mri:performance,tab:error-time}. In the performance plots, the iteration count is scaled by the ratio between the numbers of coarse- and fine-grid pixels.

\def\rhoOne{\ensuremath{10^{-4}}}
\def\rhoTwo{\ensuremath{10^{-5}}}

\begin{table}[t]
    \caption{CPU time (s) to reach $\mathrm{RP}_k=\rhoOne$ and $\mathrm{RP}_k=\rhoTwo$ for the MRI problem.}
    \label{tab:error-time}
    \centering
    \begin{tabular}{c@{\quad}@{\quad}c@{\quad}c@{\quad}c@{\quad}c@{\quad}c@{\quad}c}
        \toprule
        MRI               & \multicolumn{3}{c}{$\mathrm{RP}_k=$ \rhoOne} & \multicolumn{3}{c}{$\mathrm{RP}_k=$ \rhoTwo}
        \\
        \midrule
        Resolution        & PDPS                                         & PDMCC-E                                      & PDMCC-NE         & PDPS           & PDMCC-E          & PDMCC-NE
        \\
        \midrule
        $583\times 493$   & \num{259.09375}                              & \num{18.828125}                              & \num{29.90625}   & \num{738.5625} & \num{122.984375} & \num{142.0625}
        \\
        $2048\times 1732$ & \num{6754.328125}                            & \num{402.03125}                              & \num{542.484375} & \num{20193.1}  & \num{1816.03125} & \num{1810.1}
        \\
        \bottomrule
    \end{tabular}
\end{table}

\begin{figure}[t]
    \centering
    \begin{subfigure}{.15\textwidth}
        \centering
        \includegraphics[width=0.95\linewidth]{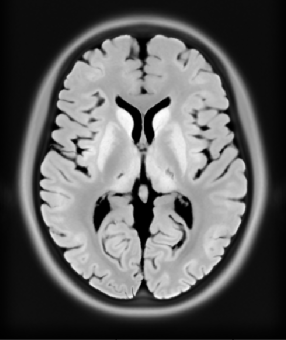}
        \caption{Original}
        \label{fig:original:image:mri:583}
    \end{subfigure}
    \begin{subfigure}{.15\textwidth}
        \centering
        \includegraphics[width=0.95\linewidth]{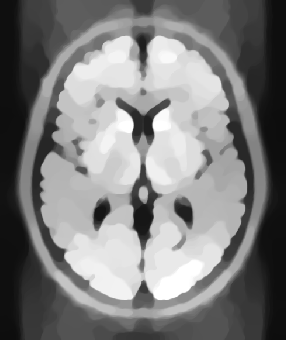}
        \caption{Solution}
        \label{fig:solution:image:mri:583}
    \end{subfigure}
    \hspace{0.5cm}
    \begin{subfigure}{.15\textwidth}
        \centering
        \includegraphics[width=0.95\linewidth]{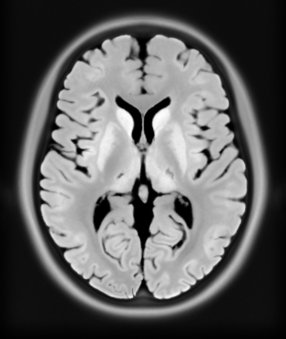}
        \caption{Original}
        \label{fig:original:image:mri:2048}
    \end{subfigure}
    \begin{subfigure}{.15\textwidth}
        \centering
        \includegraphics[width=0.95\linewidth]{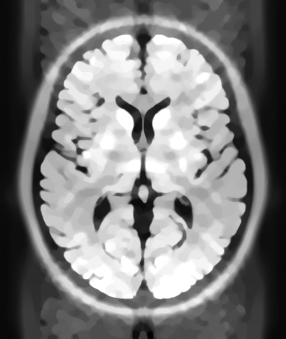}
        \caption{Solution}
        \label{fig:solution:image:mri:2048}
    \end{subfigure}
    \caption{MRI ground truth and reconstructions at relative Lagrangian gap $\mathrm{RP}_k=\rhoTwo$ for image resolutions $583\times493$ (left group)
        and $2048\times1732$ (right group).
    }
    \label{fig:mri:graphs:solutions}
\end{figure}

\begin{figure}[t]
    \centering
    \begin{subfigure}{0.48\columnwidth}
        \begin{tikzpicture}
            \begin{axis}[%
                    legend to name=leg:global,
                    legend columns=-1,
                    legend style={
                            /tikz/every even column/.append style={column sep=8pt},
                            draw=none,
                        },
                    width = \linewidth,
                    height = 0.6\linewidth,
                    axis x line*=bottom,
                    axis y line*=left,
                    xlabel={Iteration comparison \#},
                    ylabel={Lagrangian gap},
                    ymode=log,
                    font=\footnotesize,
                ]
                \addplot [opt1] table[x=iter,y=rel_ini]{results/MRI/test_583/MRI_583x493pd.txt};
                \addlegendentry{PDPS}

                \addplot [opt2] table[x=iter,y=rel_ini]{results/MRI/test_583/MRI_583x493pdmcc_ergodic.txt};
                \addlegendentry{PDMCC-ergodic}

                \addplot [opt3] table[x=iter,y=rel_ini]{results/MRI/test_583/MRI_583x493pdmcc_nonergodic.txt};
                \addlegendentry{PDMCC-non-ergodic}

                \addplot[dashed,gray!50] table {
                        1 0.0001
                        8000  0.0001
                    };
                \addplot[dashed,gray!50] table {
                        1 0.00001
                        23000  0.00001
                    };
            \end{axis}
        \end{tikzpicture}
    \end{subfigure}
    \hfill
    \begin{subfigure}{0.48\columnwidth}
        \begin{tikzpicture}
            \begin{axis}[%
                    width = \linewidth,
                    height = 0.6\linewidth,
                    axis x line*=bottom,
                    axis y line*=left,
                    xlabel={CPU Time},
                    ylabel={Lagrangian gap},
                    ymode=log,
                    font=\footnotesize,
                ]
                \addplot [opt1] table[x=cputime,y=rel_ini]{results/MRI/test_583/MRI_583x493pd.txt};

                \addplot [opt2] table[x=cputime,y=rel_ini]{results/MRI/test_583/MRI_583x493pdmcc_ergodic.txt};

                \addplot [opt3] table[x=cputime,y=rel_ini]{results/MRI/test_583/MRI_583x493pdmcc_nonergodic.txt};

                \addplot[dashed,gray!50] table {
                        0.1 0.0001
                        259 0.0001
                    };
                \addplot[dashed,gray!50] table {
                        0.1 0.00001
                        738  0.00001
                    };
            \end{axis}
        \end{tikzpicture}
    \end{subfigure}
    \\
    \begin{subfigure}{0.48\columnwidth}
        \begin{tikzpicture}
            \begin{axis}[%
                    width = \linewidth,
                    height = 0.6\linewidth,
                    axis x line*=bottom,
                    axis y line*=left,
                    xlabel={Iteration comparison \#},
                    ylabel={Lagrangian gap},
                    ymode=log,
                    font=\footnotesize,
                ]
                \addplot [opt1] table[x=iter,y=rel_ini]{results/MRI/test_2048/MRI_2048x1732pd.txt};

                \addplot [opt2] table[x=iter,y=rel_ini]{results/MRI/test_2048/MRI_2048x1732pdmcc_ergodic.txt};

                \addplot [opt3] table[x=iter,y=rel_ini]{results/MRI/test_2048/MRI_2048x1732pdmcc_nonergodic.txt};

                \addplot[dashed,gray!50] table {
                        1 0.0001
                        12000 0.0001
                    };
                \addplot[dashed,gray!50] table {
                        1 0.00001
                        30000  0.00001
                    };
            \end{axis}
        \end{tikzpicture}
    \end{subfigure}
    \hfill
    \begin{subfigure}{0.48\columnwidth}
        \begin{tikzpicture}
            \begin{axis}[%
                    width = \linewidth,
                    height = 0.6\linewidth,
                    axis x line*=bottom,
                    axis y line*=left,
                    xlabel={CPU Time},
                    ylabel={Lagrangian gap},
                    ymode=log,
                    font=\footnotesize,
                ]
                \addplot [opt1] table[x=cputime,y=rel_ini]{results/MRI/test_2048/MRI_2048x1732pd.txt};

                \addplot [opt2] table[x=cputime,y=rel_ini]{results/MRI/test_2048/MRI_2048x1732pdmcc_ergodic.txt};

                \addplot [opt3] table[x=cputime,y=rel_ini]{results/MRI/test_2048/MRI_2048x1732pdmcc_nonergodic.txt};

                \addplot[dashed,gray!50] table {
                        1 0.0001
                        8000 0.0001
                    };
                \addplot[dashed,gray!50] table {
                        1 0.00001
                        20193  0.00001
                    };
            \end{axis}
        \end{tikzpicture}
    \end{subfigure}
    \vspace{-1.5ex}
    \makebox[\textwidth][c]{\pgfplotslegendfromname{leg:global}}
    \vspace{-3.5ex}
    \caption{Relative performance measure \eqref{eq:relative:performance} versus iterations and CPU time for the MRI problem across both resolutions ($583\times 493$ in the top row and $2048\times 1732$ in the bottom row). Dashed gray lines correspond to the levels reported in \cref{tab:error-time}.}
    \label{fig:mri:performance}
\end{figure}
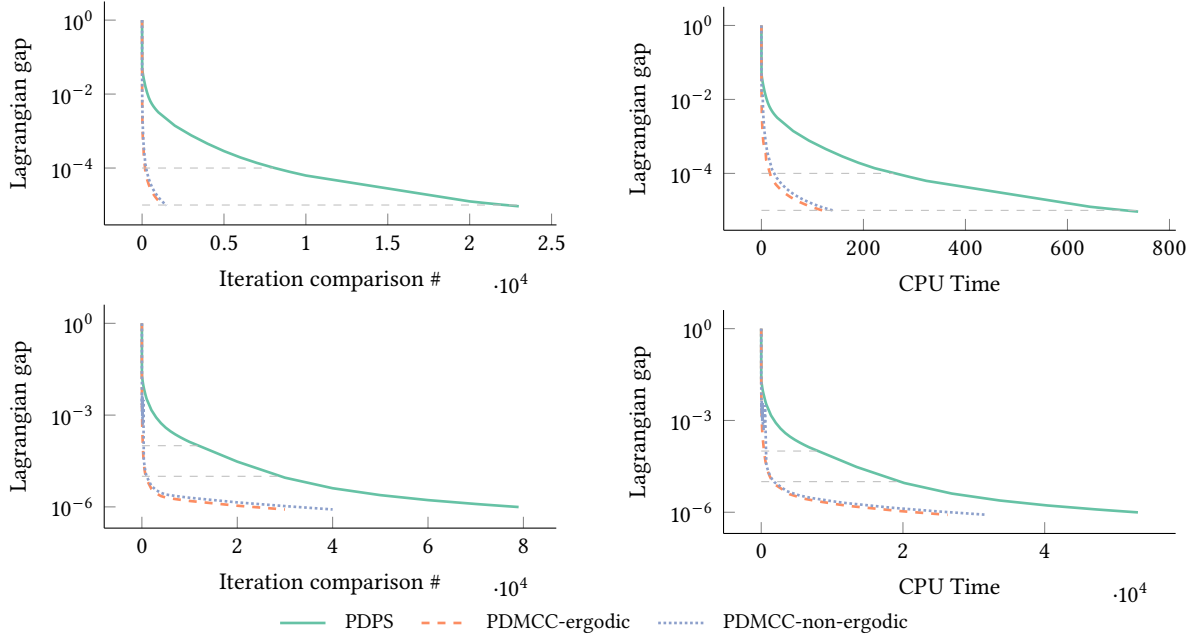

\subsection{PET}
\label{sec:numerical:pet}

In PET \cite{natterer2001mathematics}, a measurement device detects pairs of gamma photons emanating from positron-electron annihilation.
Mathematically, this process is described by the Radon transform, whereas the noise follows the Poisson distribution.
More precisely, denoting the partial discrete Radon transform by $A\in \R^{t\times n}$, then for every sampling index $p=1,\ldots,t$, the measurement $b_p \sim \operatorname{Poisson}\bigl((Ax+c)_p\bigr)$, where $c \in \R^t_+$ is a noise parameter.
The data vector $b =
    \begin{pmatrix*}
        b_1,\ldots,b_t
    \end{pmatrix*}
    ^\top \in \R^t$
is commonly known as the “sinogram”.
In the reconstruction problem \eqref{eq:modelo:imagen:general} we, thus, take
\[
    F(x) \defeq \delta_{[0,+\infty)^n}(x)
    \quad\text{and}\quad
    E(x)\defeq \iprod{\mathbf{1}_t}{Ax} - \iprod{b}{\log(Ax + c)},
\]
where the logarithm is to be understood componentwise.

The transfer and discrete gradient operators are the same as in \cref{sec:numerical:mri}.
On the coarse grid, we construct $E_H$ following \cref{ex:construction:smooth:linearistion}.
Then $L_H=0$.
We use \cref{lemma:construction:nonsmooth-polar} to construct $F_H^k\defeq \delta_{\prod_{i=1}^N \Omega_i^k}$. Denoting by $A_i\subset \{1,\ldots,n\}$ the subset of fine-grid pixel indices $j$ that contribute to the coarse pixel $i$, i.e., $[P_h^H]_{ij}\neq 0$, it gives
\[
    \Omega_i^k = \zeta_i^{k,0} + [P_h^H \partial F(\xintermed{k})]_i^\circ
    =
    \begin{cases*}
    \R,                         & if $[\xintermed{k}]_l > 0$ for all $l\in A_i$,
    \\
    [\zeta_i^{k,0},\infty)     & if there exist $l\in A_i$ such that $[\xintermed{k}]_l = 0$.
    \end{cases*}
\]

As our ground-truth image $\hat x$, we take the Shepp-Logan phantom \cite{SheppLogan1974} at the resolutions $512\times 512$ and $1024\times 1024$.
We add Poisson noise of parameter $\nu = 0.1$ to the sinogram $A\hat x$.
As the regularization parameter we take $\alpha = 0.7$.
The fine-grid step length parameters for both PDMCC and PDPS are taken as in the MRI experiments in \cref{sec:numerical:mri}, for the Lipschitz factor estimate $L = \norm{e \oslash c^2}_{\infty}\sqrt{n_x^2 + n_y^2}$. For the coarse problem, we set $\tau_H=\tau_0$, and $\sigma_H=\sigma_0/(8\tau_H)$, where $\tau_0,\sigma_0\in(0,1)$.
The dimensions of the sinogram and the PDMCC parameters are as follows:
\begin{center}
    \begin{tabular}{c c c c c c}
        \cmidrule[\heavyrulewidth]{3-6}
                         &                     & \multicolumn{2}{c}{Ergodic} & \multicolumn{2}{c}{Non-ergodic}
        \\
        \cmidrule[\lightrulewidth]{3-6}
        Resolution       & Sinogram dimensions & $m$                         & $(\tau_0,\sigma_0)$             & $m$ & $(\tau_0,\sigma_0)$
        \\
        \midrule
        $512\times512$   & $256\times 128$     & 4                           & $(0.999,\,0.99)$                & 7   & $(0.999,\,0.9)$
        \\
        $1024\times1024$ & $512\times 384$     & 2                           & $(0.37,\,0.9)$                  & 7   & $(0.999,\,0.9)$
        \\
        \bottomrule
    \end{tabular}
\end{center}

The observed data and reconstructions are in \cref{fig:pet:graphs:solutions}, while we report performance in \cref{fig:pet:performance,tab:error-time:PET}. In the performance plots, the iteration count is scaled by the ratio between the number of coarse- and fine-grid pixels.

\def\rhoThree{\ensuremath{10^{-4}}}
\def\rhoFour{\ensuremath{10^{-6}}}

\begin{table}[t]
    \caption{CPU time (s) to reach $\mathrm{RP}_k=\rhoThree$ and $\mathrm{RP}_k=\rhoFour$ for the PET problem.}
    \label{tab:error-time:PET}
    \centering
    \begin{tabular}{c@{\quad}@{\quad}c@{\quad}c@{\quad}c@{\quad}c@{\quad}c@{\quad}c}
        \toprule
        PET               & \multicolumn{3}{c}{$\mathrm{RP}_k=$ \rhoThree} & \multicolumn{3}{c}{$\mathrm{RP}_k=$ \rhoFour}
        \\
        \midrule
        Resolution        & PDPS                                           & PDMCC-E                                       & PDMCC-NE         & PDPS                & PDMCC-E             & PDMCC-NE
        \\
        \midrule
        $512\times 512$   & \num{3133.921875}                              & \num{119.265625}                              & \num{59.421875}  & \num{38005.890625}  & \num{12666.40625}   & \num{12638.9375}
        \\
        $1024\times 1024$ & \num{76331.34375}                              & \num{740.625}                                 & \num{513.078125} & \num{645403.890625} & \num{288836.328125} & \num{281258.203125}
        \\
        \bottomrule
    \end{tabular}
\end{table}

\begin{figure}[t]
    \centering
    \begin{subfigure}{.15\textwidth}
        \centering
        \includegraphics[width=0.95\linewidth]{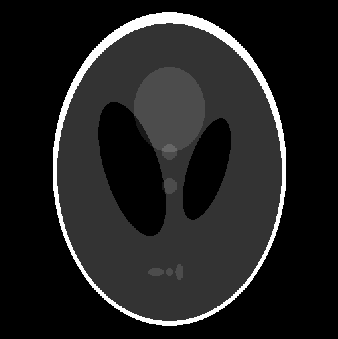}
        \caption{Original}
        \label{fig:original:image:pet:512}
    \end{subfigure}
    \begin{subfigure}{.15\textwidth}
        \centering
        \includegraphics[width=0.95\linewidth]{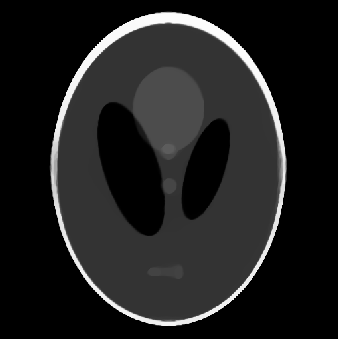}
        \caption{Solution}
        \label{fig:solution:image:pet:512}
    \end{subfigure}
    \hspace{0.5cm}
    \begin{subfigure}{.15\textwidth}
        \centering
        \includegraphics[width=0.95\linewidth]{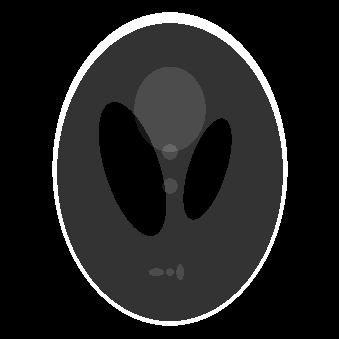}
        \caption{Original}
        \label{fig:original:image:pet:1024}
    \end{subfigure}
    \begin{subfigure}{.15\textwidth}
        \centering
        \includegraphics[width=0.95\linewidth]{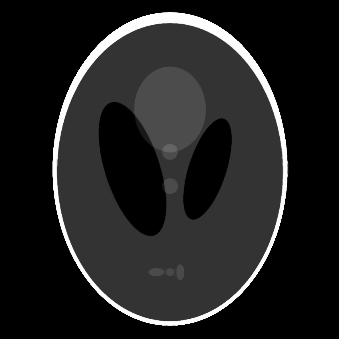}
        \caption{Solution}
        \label{fig:solution:image:pet:1024}
    \end{subfigure}
    \caption{PET ground truth and reconstructions at relative Lagrangian gap $\mathrm{RP}_k=\rhoFour$ for image resolutions $512\times512$ (left group)
        and $1024\times1024$ (right group).
    }
    \label{fig:pet:graphs:solutions}
\end{figure}

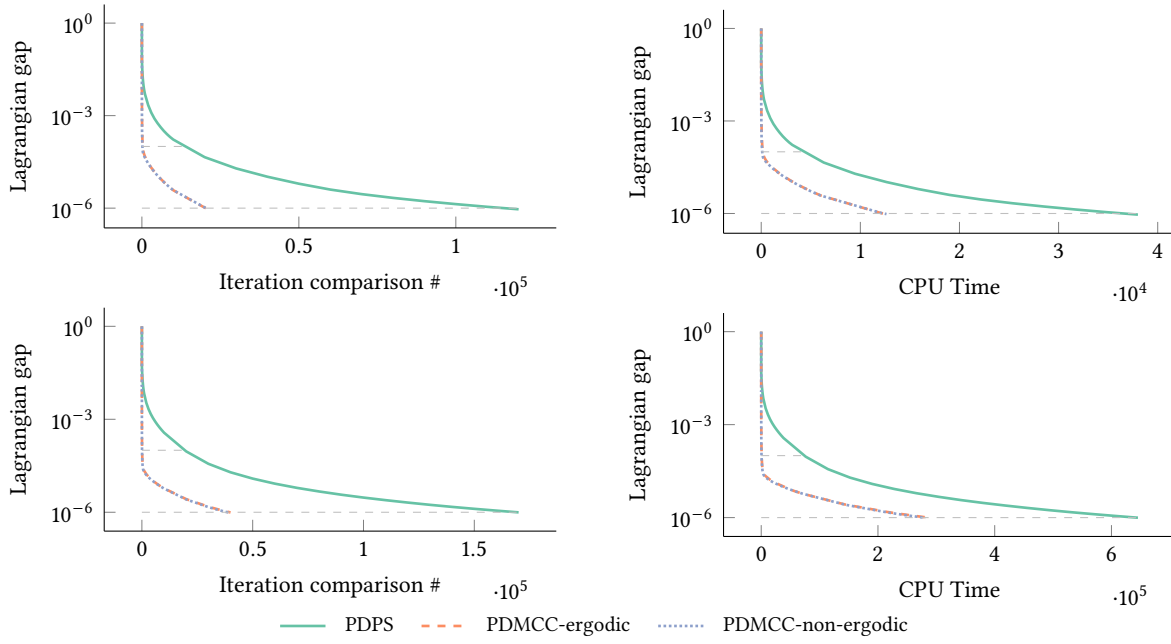
\begin{figure}[t]
    \centering
    \begin{subfigure}{0.48\columnwidth}
        \begin{tikzpicture}
            \begin{axis}[%
                    legend to name=leg:global,
                    legend columns=-1,
                    legend style={
                            /tikz/every even column/.append style={column sep=8pt},
                            draw=none,
                        },
                    width = \linewidth,
                    height = 0.6\linewidth,
                    axis x line*=bottom,
                    axis y line*=left,
                    xlabel={Iteration comparison \#},
                    ylabel={Lagrangian gap},
                    ymode=log,
                    font=\footnotesize,
                ]
                \addplot [opt1] table[x=iter,y=rel_ini]{results/PET/test_512/PET_512x512pd.txt};
                \addlegendentry{PDPS}

                \addplot [opt2] table[x=iter,y=rel_ini]{results/PET/test_512/PET_512x512pdmcc_ergodic.txt};
                \addlegendentry{PDMCC-ergodic}

                \addplot [opt3] table[x=iter,y=rel_ini]{results/PET/test_512/PET_512x512pdmcc_nonergodic.txt};
                \addlegendentry{PDMCC-non-ergodic}

                \addplot[dashed,gray!50] table {
                        1 0.0001
                        15000 0.0001
                    };
                \addplot[dashed,gray!50] table {
                        1 0.000001
                        120000  0.000001
                    };
            \end{axis}
        \end{tikzpicture}
    \end{subfigure}
    \hfill
    \begin{subfigure}{0.48\columnwidth}
        \begin{tikzpicture}
            \begin{axis}[%
                    width = \linewidth,
                    height = 0.6\linewidth,
                    axis x line*=bottom,
                    axis y line*=left,
                    xlabel={CPU Time},
                    ylabel={Lagrangian gap},
                    ymode=log,
                    font=\footnotesize,
                ]
                \addplot [opt1] table[x=cputime,y=rel_ini]{results/PET/test_512/PET_512x512pd.txt};

                \addplot [opt2] table[x=cputime,y=rel_ini]{results/PET/test_512/PET_512x512pdmcc_ergodic.txt};

                \addplot [opt3] table[x=cputime,y=rel_ini]{results/PET/test_512/PET_512x512pdmcc_nonergodic.txt};

                \addplot[dashed,gray!50] table {
                        0.1 0.0001
                        4200 0.0001
                    };
                \addplot[dashed,gray!50] table {
                        0.1 0.000001
                        38004  0.000001
                    };
            \end{axis}
        \end{tikzpicture}
    \end{subfigure}
    \\
    \begin{subfigure}{0.48\columnwidth}
        \begin{tikzpicture}
            \begin{axis}[%
                    width = \linewidth,
                    height = 0.6\linewidth,
                    axis x line*=bottom,
                    axis y line*=left,
                    xlabel={Iteration comparison \#},
                    ylabel={Lagrangian gap},
                    ymode=log,
                    font=\footnotesize,
                ]
                \addplot [opt1] table[x=iter,y=rel_ini]{results/PET/test_1024/PET_1024x1024pd.txt};

                \addplot [opt2] table[x=iter,y=rel_ini]{results/PET/test_1024/PET_1024x1024pdmcc_ergodic.txt};

                \addplot [opt3] table[x=iter,y=rel_ini]{results/PET/test_1024/PET_1024x1024pdmcc_nonergodic.txt};

                \addplot[dashed,gray!50] table {
                        1 0.0001
                        20000 0.0001
                    };
                \addplot[dashed,gray!50] table {
                        1 0.000001
                        170000  0.000001
                    };
            \end{axis}
        \end{tikzpicture}
    \end{subfigure}
    \hfill
    \begin{subfigure}{0.48\columnwidth}
        \begin{tikzpicture}
            \begin{axis}[%
                    width = \linewidth,
                    height = 0.6\linewidth,
                    axis x line*=bottom,
                    axis y line*=left,
                    xlabel={CPU Time},
                    ylabel={Lagrangian gap},
                    ymode=log,
                    font=\footnotesize,
                ]
                \addplot [opt1] table[x=cputime,y=rel_ini]{results/PET/test_1024/PET_1024x1024pd.txt};

                \addplot [opt2] table[x=cputime,y=rel_ini]{results/PET/test_1024/PET_1024x1024pdmcc_ergodic.txt};

                \addplot [opt3] table[x=cputime,y=rel_ini]{results/PET/test_1024/PET_1024x1024pdmcc_nonergodic.txt};

                \addplot[dashed,gray!50] table {
                        4 0.0001
                        76330 0.0001
                    };
                \addplot[dashed,gray!50] table {
                        4 0.000001
                        645404  0.000001
                    };
            \end{axis}
        \end{tikzpicture}
    \end{subfigure}
    \vspace{-1.5ex}
    \makebox[\textwidth][c]{\pgfplotslegendfromname{leg:global}}
    \vspace{-3.5ex}
    \caption{Relative performance measure \eqref{eq:relative:performance} versus iterations and CPU time for the PET problem across both resolutions ($512\times 512$ in the top row and $1024\times 1024$ in the bottom row). Dashed gray lines correspond to the levels reported in \cref{tab:error-time:PET}.}
    \label{fig:pet:performance}
\end{figure}

\subsection{EIT}
\label{sec:numerical:eit}

We now take in \eqref{eq:modelo:imagen:general}
\[
    E(x) \defeq \frac{1}{2} \sum_{i=1}^{d} \norm{I_i(x)-\mathscr{I}_i}^2,\quad F(x) = \delta_{[x_{\min}, x_{\max}]^n}(x),
\]
where, for a given conductivity $x \in L^2(\Omega)$, $I_i(x) \in \R^{d}$ are simulated electrical currents at $d \in \N$ electrodes on the boundary of the domain $\Omega$, when the same electodes are excited with the electrical potentials $U_i \in \R^{d}$.
The measured currents are $\mathscr{I}_i \in \R^d$.
Multiple measurements $i=1,\ldots,N$ are made.
The relationship is governed by the \emph{Complete Electrode Model} (CEM) partial differential equation (PDE), \cite{cheng1989electrode}.
For further details on our specific approach, see \cite{tuomov2024tracking}.

We work in the circular domain $\Omega=B(0,r)$ with $r=15$, equipped with $d=16$ equally spaced boundary electrodes. The coarse mesh $\mathcal{T}_H$ has 4897 nodes and 9024 elements, while the fine mesh $\mathcal{T}_h$ is obtained by uniform refinement, yielding 18817 nodes and 36096 elements.
We define the primal prolongation $P_H^h:X_H\hookrightarrow X_h$ as the canonical inclusion, and the dual restriction $D_h^H:Y_h\to Y_H$ by as the average over the four fine elements contained in each coarse element, i.e., in stencil notation,
$
    D_h^H=\frac{1}{4}
    \begin{pmatrix}
        1 & 1 & 1 & 1
    \end{pmatrix}
    .
$
We set $K=\mathscr{M}_h\grad_h$ and $K_H=\mathscr{M}_H\grad_H$, where $\mathscr{M}_h$ and $\mathscr{M}_H$ denote the corresponding dual mass matrices. As in \cref{sec:numerical:pet}, we construct the coarse objective $E_H$ according to \cref{ex:construction:smooth:ideal}.
The construction of $F_H$ is also similar to \cref{sec:numerical:pet}.
The Lipschitz constant of $\grad E$ cannot be computed explicitly \cite{tuomov2024tracking}, unlike in \cref{sec:numerical:mri,sec:numerical:pet}.
Based on dynamic estimation from initial experiments, we use $L \approx 18.2036$ and $\norm{K} \approx 0.03041$. This yields
$\tau=0.675/L \approx 0.03708$ y $\sigma = 0.07/(\norm{K}^2 \tau) \approx 2041.452$.
Meanwhile, both the ergodic and non-ergodic variants perform $m=5$ coarse iterations with $\tau_H=0.12$ and $\sigma_H=0.9/(\norm{K_H}^2\tau_H)$ for $\norm{K_H} \approx 0.2416$.

The observed data and reconstructed images are shown in \cref{fig:eit:graphs:solutions}, while performance comparisons are presented in \cref{fig:eit:performance} and \cref{tab:error-time:EIT}.
\def\rhoFive{\ensuremath{3\cdot 10^{-4}}}
\def\rhoSix{\ensuremath{5\cdot 10^{-5}}}

\begin{table}[t]
    \caption{CPU time (s) to reach $\mathrm{RP}_k=\rhoFive$ and $\mathrm{RP}_k=\rhoSix$ for EIT.}
    \label{tab:error-time:EIT}
    \centering
    \begin{tabular}{c@{\quad}@{\quad}c@{\quad}c@{\quad}c@{\quad}c@{\quad}c@{\quad}c}
        \toprule
        EIT        & \multicolumn{3}{c}{$\mathrm{RP}_k=$ \rhoFive} & \multicolumn{3}{c}{$\mathrm{RP}_k=$ \rhoSix}
        \\
        \midrule
        Fine nodes & PDPS                                          & PDMCC-E                                      & PDMCC-NE          & PDPS               & PDMCC-E            & PDMCC-NE
        \\
        \midrule
        $18817$    & \num{24866.890625}                            & \num{5069.453125}                            & \num{5070.828125} & \num{84098.515625} & \num{18772.734375} & \num{19526.453125}
        \\
        \bottomrule
    \end{tabular}
\end{table}

\begin{figure}[t]
    \centering
    \begin{subfigure}{.195\textwidth}
        \centering
        \includegraphics[width=0.95\linewidth]{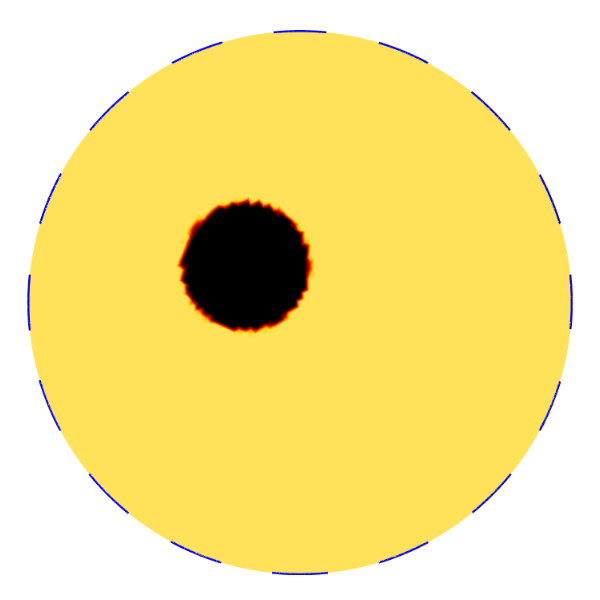}
        \caption{Original}
        \label{fig:original:image:eit}
    \end{subfigure}
    \hfil
    \begin{subfigure}{.195\textwidth}
        \centering
        \includegraphics[width=0.95\linewidth]{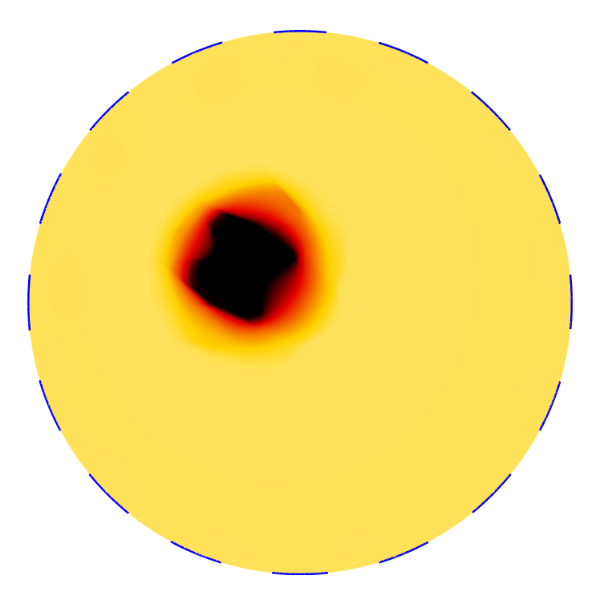}
        \caption{Solution}
        \label{fig:solution:image:eit}
    \end{subfigure}
    \caption{EIT ground-truth and reconstruction at relative primal objective  function $\mathrm{RP}_k=\ensuremath{4\cdot 10^{-5}}$ for $18817$ node grid.
    }
    \label{fig:eit:graphs:solutions}
\end{figure}

\begin{figure}[t]
    \centering
    \begin{subfigure}{0.48\columnwidth}
        \begin{tikzpicture}
            \begin{axis}[%
                    legend to name=leg:global,
                    legend columns=-1,
                    legend style={
                            /tikz/every even column/.append style={column sep=8pt},
                            draw=none,
                        },
                    width = \linewidth,
                    height = 0.6\linewidth,
                    axis x line*=bottom,
                    axis y line*=left,
                    xlabel={Iteration comparison \#},
                    ylabel={Primal objective},
                    ymode=log,
                    font=\footnotesize,
                ]
                \addplot [opt1] table[x=iter,y=rel_po]{results/EIT/pd_euclidian.txt};
                \addlegendentry{PDPS}

                \addplot [opt2] table[x=iter,y=rel_po]{results/EIT/pdmcc_ergodic_euclidian.txt};
                \addlegendentry{PDMCC-ergodic}

                \addplot [opt3] table[x=iter,y=rel_po]{results/EIT/pdmcc_nonergodic_euclidian.txt};
                \addlegendentry{PDMCC-non-ergodic}

                \addplot[dashed,gray!50] table {
                        0 0.0003
                        4000 0.0003
                    };
                \addplot[dashed,gray!50] table {
                        0 0.00005
                        21000  0.00005
                    };
            \end{axis}
        \end{tikzpicture}
    \end{subfigure}
    \hfill
    \begin{subfigure}{0.48\columnwidth}
        \begin{tikzpicture}
            \begin{axis}[%
                    width = \linewidth,
                    height = 0.6\linewidth,
                    axis x line*=bottom,
                    axis y line*=left,
                    xlabel={CPU Time},
                    ylabel={Primal objective},
                    ymode=log,
                    font=\footnotesize,
                ]
                \addplot [opt1] table[x=cputime,y=rel_po]{results/EIT/pd_euclidian.txt};

                \addplot [opt2] table[x=cputime,y=rel_po]{results/EIT/pdmcc_ergodic_euclidian.txt};

                \addplot [opt3] table[x=cputime,y=rel_po]{results/EIT/pdmcc_nonergodic_euclidian.txt};

                \addplot[dashed,gray!50] table {
                        0 0.0003
                        24900 0.0003
                    };
                \addplot[dashed,gray!50] table {
                        0 0.00005
                        84098  0.00005
                    };
            \end{axis}
        \end{tikzpicture}
    \end{subfigure}
    \vspace{-1.5ex}
    \makebox[\textwidth][c]{\pgfplotslegendfromname{leg:global}}
    \vspace{-3.5ex}
    \caption{Relative performance measure \eqref{eq:relative:performance} versus iterations and CPU time for the EIT problem. Dashed gray lines correspond to the levels reported in \cref{tab:error-time:EIT}.}
    \label{fig:eit:performance}
\end{figure}
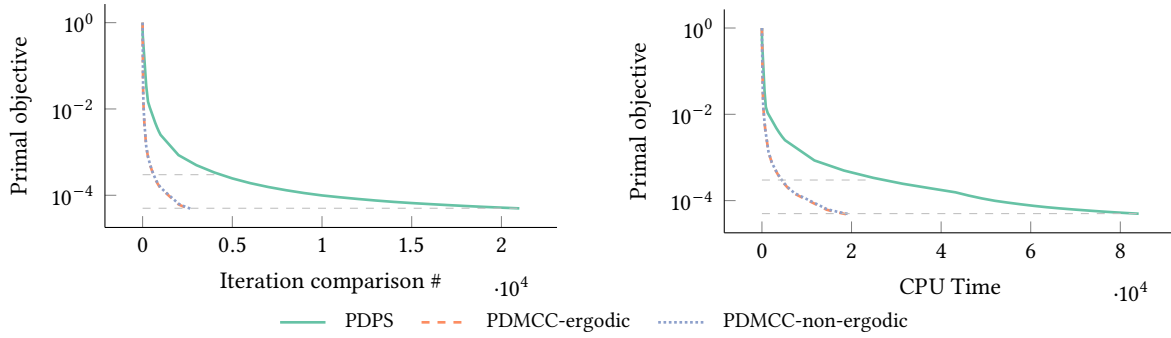

\subsection{Conclusions}
\label{sec:numerical:conclusions}

\Cref{fig:mri:performance,fig:pet:performance,fig:eit:performance} illustrate the convergence behavior of PDMCC, which is consistent with the $O(1/N)$ convergence rate established in \cref{cor:ergodic:convergence:PDMCC}. The same behavior is observed for PDPS, in agreement with \cref{thm:basic:PDPS-inequality}.
Furthermore, \cref{tab:error-time,tab:error-time:PET,tab:error-time:EIT} show that the proposed PDMCC variants require substantially less CPU time than PDPS to reach the same reference value of $\mathrm{RP}_k$.
Specifically, the reduction ranges from 81\% to 93\% for MRI, from 67\% to 98\% for PET, and from 76\% to 79\% for EIT.
These results confirm the theoretical results presented in \cref{sec:convergence}, indeed, show much faster convergence than that of the reference PDPS.
In conclusion, our proposed method appears to provide the advantages that multigrid methods generally have. In future research, it would be desirable to produce a more theoretical analysis of the reduced computational cost.

\bibliographystyle{jnsao}

\begin{thebibliography}{10}

\bibitem{ang2024mgprox}
A{.\nobreak\kern 0.33333em}Ang, H{.\nobreak\kern 0.33333em}De~Sterck, and
  S{.\nobreak\kern 0.33333em}Vavasis, {MGProx}: {A} nonsmooth multigrid
  proximal gradient method with adaptive restriction for strongly convex
  optimization, \emph{SIAM J. Optim.} 34 (2024),  2788--2820.

\bibitem{belzunce2018high}
M.\,A{.\nobreak\kern 0.33333em}Belzunce, High-Resolution Heterogeneous Digital
  {PET} {[18F]FDG} Brain Phantom based on the {BigBrain} Atlas, 2018,
  \href{https://dx.doi.org/10.5281/zenodo.1190598}{\nolinkurl{doi:10.5281/zenodo.1190598}}.

\bibitem{brandt1977multilevel}
A{.\nobreak\kern 0.33333em}Brandt, Multi-Level Adaptive Solutions to
  Boundary-Value Problems, \emph{Mathematics of Computation} 31 (1977),
  333--390,
  \href{https://dx.doi.org/10.1090/S0025-5718-1977-0431719-X}{\nolinkurl{doi:10.1090/s0025-5718-1977-0431719-x}}.

\bibitem{briggs2000multigrid}
W.\,L{.\nobreak\kern 0.33333em}Briggs, V.\,E{.\nobreak\kern 0.33333em}Henson,
  and S.\,F{.\nobreak\kern 0.33333em}McCormick, \emph{A multigrid tutorial},
  SIAM, 2000.

\bibitem{chambolle2004algorithm}
A{.\nobreak\kern 0.33333em}Chambolle, An algorithm for total variation
  minimization and applications, \emph{Journal of Mathematical imaging and
  vision} 20 (2004),  89--97.

\bibitem{chambolle2011first}
A{.\nobreak\kern 0.33333em}Chambolle and T{.\nobreak\kern 0.33333em}Pock, A
  first-order primal-dual algorithm for convex problems with applications to
  imaging, \emph{Journal of mathematical imaging and vision} 40 (2011),
  120--145.

\bibitem{cheng1989electrode}
K.\,S{.\nobreak\kern 0.33333em}Cheng, D{.\nobreak\kern 0.33333em}Isaacson,
  J.\,C{.\nobreak\kern 0.33333em}Newell, and D.\,G{.\nobreak\kern
  0.33333em}Gisser, Electrode models for electric current computed tomography,
  \emph{IEEE Transactions on Biomedical Engineering} 36 (1989),  918--924.

\bibitem{tuomov-pdex2nlpdhgm}
C{.\nobreak\kern 0.33333em}Clason and T{.\nobreak\kern 0.33333em}Valkonen,
  Primal-dual extragradient methods for nonlinear nonsmooth {PDE}-constrained
  optimization, \emph{SIAM J. Optim.} 27 (2017),  1313--1339,
  \href{https://dx.doi.org/10.1137/16M1080859}{\nolinkurl{doi:10.1137/16m1080859}},
  \href{https://arxiv.org/abs/1606.06219}{\nolinkurl{arXiv:1606.06219}}.

\bibitem{clason2020introduction}
C{.\nobreak\kern 0.33333em}Clason and T{.\nobreak\kern 0.33333em}Valkonen,
  \emph{Introduction to Nonsmooth Analysis and Optimization}, MOS-SIAM Series
  on Optimization, SIAM, 2026,
  \href{https://dx.doi.org/10.1137/1.9781611978995}{\nolinkurl{doi:10.1137/1.9781611978995}}.

\bibitem{condat2013primaldual}
L{.\nobreak\kern 0.33333em}Condat, A Primal--Dual Splitting Method for Convex
  Optimization Involving Lipschitzian, Proximable and Linear Composite Terms,
  \emph{J. Optim. Theory Appl.} 158 (2013),  460--479,
  \href{https://dx.doi.org/10.1007/s10957-012-0245-9}{\nolinkurl{doi:10.1007/s10957-012-0245-9}}.

\bibitem{tuomov2024online-eit}
N{.\nobreak\kern 0.33333em}Dizon, J{.\nobreak\kern 0.33333em}Jauhiainen, and
  T{.\nobreak\kern 0.33333em}Valkonen, Online optimisation for dynamic
  electrical impedance tomography, \emph{Inv. Prob.} 41 (2025),  055005,
  \href{https://dx.doi.org/10.1088/1361-6420/adcb66}{\nolinkurl{doi:10.1088/1361-6420/adcb66}},
  \href{https://arxiv.org/abs/2412.12944}{\nolinkurl{arXiv:2412.12944}}.

\bibitem{tuomov2024tracking}
N{.\nobreak\kern 0.33333em}Dizon and T{.\nobreak\kern 0.33333em}Valkonen,
  Differential estimates for fast first-order multilevel nonconvex
  optimisation, 2024,
  \href{https://arxiv.org/abs/2412.01481}{\nolinkurl{arXiv:2412.01481}}.
\newblock submitted.

\bibitem{gabay1983chapter}
D{.\nobreak\kern 0.33333em}Gabay, Chapter ix applications of the method of
  multipliers to variational inequalities, in \emph{Studies in mathematics and
  its applications}, volume~15, Elsevier, 1983,  299--331.

\bibitem{guerra2025multigrid}
F{.\nobreak\kern 0.33333em}Guerra and T{.\nobreak\kern 0.33333em}Valkonen,
  Multigrid methods for total variation, in \emph{International Conference on
  Scale Space and Variational Methods in Computer Vision}, 2025,  3--16.

\bibitem{guerra2026pdmultigrid-codes}
F{.\nobreak\kern 0.33333em}Guerra and T{.\nobreak\kern 0.33333em}Valkonen,
  Multigrid methods for nonsmooth optimization, 2026,
  \href{https://dx.doi.org/10.5281/zenodo.21731057}{\nolinkurl{doi:10.5281/zenodo.21731057}}.
\newblock Software on Zenodo.

\bibitem{he2014convergence}
B{.\nobreak\kern 0.33333em}He, Y{.\nobreak\kern 0.33333em}You, and
  X{.\nobreak\kern 0.33333em}Yuan, On the Convergence of Primal-Dual Hybrid
  Gradient Algorithm, \emph{SIAM J. Imaging Sci.} 7 (2014),  2526--2537,
  \href{https://dx.doi.org/10.1137/140963467}{\nolinkurl{doi:10.1137/140963467}}.

\bibitem{jauhiainen2025online-eit-codes}
J{.\nobreak\kern 0.33333em}Jauhiainen, N{.\nobreak\kern 0.33333em}Dizon,
  T{.\nobreak\kern 0.33333em}Valkonen, and Y{.\nobreak\kern 0.33333em}Nabou,
  Online Optimisation Codes for Dynamic Electrical Impedance Tomography, 2026,
  \href{https://dx.doi.org/10.5281/zenodo.19154746}{\nolinkurl{doi:10.5281/zenodo.19154746}}.
\newblock Software.

\bibitem{malitsky2018first}
Y{.\nobreak\kern 0.33333em}Malitsky and T{.\nobreak\kern 0.33333em}Pock, A
  first-order primal-dual algorithm with linesearch, \emph{SIAM Journal on
  Optimization} 28 (2018),  411--432.

\bibitem{nash2000multigrid}
S.\,G{.\nobreak\kern 0.33333em}Nash, A multigrid approach to discretized
  optimization problems, \emph{Optim. Methods. Software} 14 (2000),  99--116.

\bibitem{natterer2001mathematics}
F{.\nobreak\kern 0.33333em}Natterer, \emph{The mathematics of computerized
  tomography}, SIAM, 2001.

\bibitem{nocedal2006numerical}
J{.\nobreak\kern 0.33333em}Nocedal and S{.\nobreak\kern 0.33333em}Wright,
  \emph{Numerical optimization}, Springer Science \& Business Media, 2006.

\bibitem{parpas2017multilevel}
P{.\nobreak\kern 0.33333em}Parpas, A multilevel proximal gradient algorithm for
  a class of composite optimization problems, \emph{SIAM J. Optim.} 39 (2017),
  681--701.

\bibitem{pock2009algorithm}
T{.\nobreak\kern 0.33333em}Pock, D{.\nobreak\kern 0.33333em}Cremers,
  H{.\nobreak\kern 0.33333em}Bischof, and A{.\nobreak\kern 0.33333em}Chambolle,
  An algorithm for minimizing the Mumford-Shah functional, in \emph{2009 IEEE
  12th International Conference on Computer Vision}, 2009,  1133--1140.

\bibitem{SheppLogan1974}
L.\,A{.\nobreak\kern 0.33333em}Shepp and B.\,F{.\nobreak\kern 0.33333em}Logan,
  The Fourier Reconstruction of a Head Section, \emph{IEEE Transactions on
  Nuclear Science} NS-21 (1974),  21--43.

\bibitem{valkonen2017preconditioned}
T{.\nobreak\kern 0.33333em}Valkonen, Preconditioned {ADMM} with nonlinear
  operator constraint, in \emph{System Modeling and Optimization: 27th IFIP TC
  7 Conference, CSMO 2015, Sophia Antipolis, France, June 29-July 3, 2015,
  Revised Selected Papers}, volume 494, 2017,  117.

\bibitem{tuomov-proxtest}
T{.\nobreak\kern 0.33333em}Valkonen, Testing and non-linear preconditioning of
  the proximal point method, \emph{Appl. Math. Optim.} 82 (2020),
  \href{https://dx.doi.org/10.1007/s00245-018-9541-6}{\nolinkurl{doi:10.1007/s00245-018-9541-6}}.

\bibitem{valkonen-tracking-codes}
T{.\nobreak\kern 0.33333em}Valkonen, Codes for Differential Estimates for Fast
  First-Order Multilevel Nonconvex Optimisation, 2026,
  \href{https://dx.doi.org/10.5281/zenodo.19154665}{\nolinkurl{doi:10.5281/zenodo.19154665}}.
\newblock Software on Zenodo.

\bibitem{vu2013splitting}
B.\,C{.\nobreak\kern 0.33333em}V{\~{u}}, A splitting algorithm for dual
  monotone inclusions involving cocoercive operators, \emph{Advances in
  Computational Mathematics} 38 (2013),  667--681,
  \href{https://dx.doi.org/10.1007/s10444-011-9254-8}{\nolinkurl{doi:10.1007/s10444-011-9254-8}}.

\end{thebibliography}

\end{document}